\documentclass[11pt]{article}
\usepackage{multirow}
\usepackage{graphicx}
\usepackage{tabularx}
\usepackage{booktabs}
\usepackage{amsmath, amssymb, amscd, amsthm, mathtools, amsfonts, mathrsfs, tikz-cd, bbm, enumitem}
\usepackage[mathscr]{euscript}
\usepackage[utf8]{inputenc}
\usepackage[english]{babel}
\usepackage{mathbbol}
\usepackage{xcolor}
\usepackage[title]{appendix}
\usepackage{tikz}
\usepackage{tikz-cd}
\usepackage{bold-extra}
\usepackage{xcolor}
\usepackage{graphicx}
\usepackage{color}
\usepackage{epsfig}
\usetikzlibrary{positioning}
\usepackage{url}
\usepackage[colorlinks=true,citecolor=blue,urlcolor=blue,linkcolor=blue]{hyperref}
\usepackage{setspace}
\usepackage{geometry}
\usepackage{graphics}
\usepackage{color}
\usepackage{adjustbox}
\usepackage{color,soul}
\usepackage{color}
\usetikzlibrary{matrix}
\usepackage{color}
\usepackage{fancyhdr}
\usepackage{pdfpages}
\usepackage{color}
\usetikzlibrary{positioning}
\usepackage[colorlinks=true,citecolor=blue,urlcolor=blue,linkcolor=blue]{hyperref}
\usepackage{graphicx}
\usepackage{url}
\usepackage{epsfig}
\def\scfig #1 #2 {\resizebox{#2}{!}{\includegraphics{#1}}}
\usepackage{graphics}
\usepackage{color,soul}
\usepackage{tikz-cd}
\usepackage{tikz}
\usetikzlibrary{matrix}
\usepackage{amsmath,amsthm,amscd,amsfonts,amssymb,euscript,amsthm,amsfonts,amssymb,amscd}
\usepackage{graphics}
\usepackage{tikz-cd}
\usepackage{pdfpages}
\usepackage{mathtools}

\usepackage{color}
\usepackage{hyperref} \hypersetup{colorlinks=true,linkcolor=blue, citecolor=blue}
\usepackage{fancyhdr}
\usepackage{fullpage}
\usepackage{pdfpages}
\usepackage{mathtools}
\newcommand\node[2]{\overset{#1}{\underset{#2}{\bullet}}}
\usepackage{frcursive}
\usepackage{latexsym}
\usepackage{color}
\theoremstyle{plain}

\numberwithin{equation}{section}

\newtheorem{thm}[subsection]{Theorem}
\newtheorem{prop}[subsection]{Proposition}

\newtheorem{lema}[subsection]{Lemma}

\theoremstyle{definition}
\newtheorem{rema}[subsection]{Remark}
\newtheorem{defe}[subsection]{Definition}

\newcommand {\OO}{{\mathcal O}}

\newcommand {\J}{{\mathcal J}}

\def\gr{\operatorname{gr}}

\def\Ker{\operatorname{Ker}}

\def\dim{\operatorname{dim}}

\def\Spec{\operatorname{Spec}}

\def\dim{\operatorname{dim}}

\def\co-dim{\operatorname{co-dim}}

\title{An Equisingular Specialization of the Compactified Jacobian and its applications}
\author{ Sourav Das, A.J. Parameswaran and Subham Sarkar}

\begin{document}
\maketitle
\begin{abstract}
For any positive integer $k$, let $X_k$ be a projective irreducible nodal curve with $k$ nodes. We show that the Betti numbers and the mixed Hodge numbers of the compactified Jacobian $\overline{J_{k}}$ of an irreducible nodal curve $X_k$ with $k$ nodes are the same as the Betti numbers and the mixed Hodge numbers of $J_0\times R^k$, where $J_0$ is the Jacobian of the normalisation of the irreducible nodal curve and $R$ denotes the rational nodal curve with one node. We prove it by constructing a topologically locally trivial family of projective varieties which contains both $\overline{J_{k}}$ and $J_0\times R^k$ as fibres.
\end{abstract}

\tableofcontents

\section{\textbf{Introduction}}
Given a smooth projective algebraic curve $C$, one can associate a principally polarised Abelian variety $J_C$, called the Jacobian of $C$. It is the moduli of isomorphism classes of degree $0$ line bundles on the curve $C$. It is well-known that the Betti numbers of $J_C$ are given by $\wedge^{\bullet} \mathbb C^g$, where $g$ is the genus of the curve. Studying the moduli of line bundles on a stable nodal curve is also natural because the boundary of the Deligne-Momford compactification $\overline{M_g}$ of the moduli of curves consists of stable nodal curves.  However, the moduli of degree $0$ line bundles on a nodal curve is generally not compact and is called generalized Jacobian. A compactification of the generalized Jacobian of a nodal curve can be constructed using geometric invariant theory by choosing a polarization on the nodal curve (\cite{11}, \cite{4}). It is also an active research are to construct and study suitable universal compactified Jacobian over $\overline{M_g}$ (\cite{P, PT}). One of the many reasons for studying compactified Jacobian is its relation to the theory of Higgs bundles on curves. Some singular fibres of the Hitchin map can be described as the compactified Jacobian of some nodal curves using the so-called spectral correspondence (\cite{BNR}). On the other hand, the Langlands correspondence in the context of the Higgs bundles predicts interesting derived equivalence of different fine compactified Jacobians of a nodal curve (\cite{MRV I, MRV II}).  Compactified Jacobians also provide good examples where one can test predictions or conjectures. They are also related to the Hilbert schemes of points on nodal curves and are therefore useful in studying Hilbert schemes as well \cite{MS, MSV, MY}. In these articles, they establish a version of Macdonald's formula for integral curves with planar singularities. They show that the Betti numbers of the Hilbert scheme of the curve can be expressed as a direct sum of the shifted graded pieces of the perverse filtration on the compactified Jacobian of the curve. 

We should mention a related work \cite{Pi} on the computation of Betti numbers of the compactified Jacobians of uni-branched rational curves with some special type of singularities. 

\

In this paper, we compute the Betti numbers and the mixed Hodge numbers of the compactified Jacobian of irreducible nodal curves. 

\

Notation: For any positive integer $k$, let $X_k$ denote any irreducible nodal curve of genus $g$. Let us denote its normalization by $q_k: X_0\rightarrow X_k$. Let us denote the nodes of $X_k$ by $\{y_1,\dots, y_k\}$ and the inverse image of the node $y_i$ under the normalization map by $\{x_i, z_i\}$ for every $i=1,\dots, k$. We fix such a nodal curve $X_k$. We denote its compactified Jacobian by $\overline{J}_k$ and its normalization by $\widetilde{J_k}$. 

\

In general, the variety $\overline{J_k}$ has the product of normal crossing singularities. Therefore, its normalization is a smooth variety. Let us denote it by $\widetilde{J_k}$. A direct way to study the geometry of $\overline{J_k}$ is to study how to construct back $\overline{J_k}$ from its normalization using the following diagram.
\begin{equation}\label{Diago}
\begin{tikzcd}
& \widetilde{J_k}\arrow{dl} \arrow{dr}\\
\overline{J_k} && J_{_{0}},
\end{tikzcd}
\end{equation}
where  the left arrow is the normalization of $\overline{J_k}$ and the right arrow is a fiber-product of $k$ many $\mathbb P^1$-bundles over $J_{0}$. The precise description of the right map is as follows. Let us fix a Poincar\'e bundle $\mathcal P$ over $X_0\times J_{_{0}}$ and let $\mathcal{P}_{x_j}$ denote the restriction of the Poincar\'e bundle $\mathcal{P}$ on $\{x_i\}\times J_0$. Similarly, let $\mathcal{P}_{z_i}$ denote the restriction of the Poincar\'e bundle $\mathcal{P}$ on $\{z_i\}\times J_0$. Then, one can show that

\[
\widetilde{ J}_k\cong \mathbb P_1\times_{J_0} \mathbb P_2\times_{J_0}\dots \times_{J_0} \mathbb P_k,
\]

where $\mathbb P_i:=\mathbb P(\mathcal P_{x_i}\oplus \mathcal P_{z_i})$ is the projective bundle over $J_0$ for $i=1,\dots,k$ (\cite{4} , \cite{11}).

For each $i=1,\dots, k$, there is a pair of divisors $\{D_i, D'_i\}$ on $\widetilde{J}_k$ corresponding to the two natural quotients of the vector bundles $\mathcal P_{x_i}\oplus \mathcal P_{z_i}$. All these divisors can also be described as suitable fiber products of $\mathbb P^1$-bundles over $J_0$. Moreover, there are $k$-many "twisted" isomorphisms $\tau_i: D_i\to D'_i$. They are called twisted isomorphisms because they do not commute with the projection morphisms $\widetilde{J}_k\rightarrow J_0$ (see \cite{11}). The compactified Jacobian $\overline{J_k}$ is a categorical quotient of $\widetilde{ J}_k$ under the identifications given by the isomorphisms $\{\tau_i\}^k_{i=1}$.  

\

\textbf{Problem:} Suppose we want to compute some invariant (e.g. the mixed Hodge numbers of the cohomology groups) of $\overline{J_k}$. A natural strategy would be first compute it for $J_0$ and then for $\widetilde{J_k}$ using the projective bundle description (see \ref{Diago}) of the map $\widetilde{J_k}\to J_0$ and then use the left map in the diagram \ref{Diago} to compute the invariant for $J_k$. But since the map $\widetilde{J_k}\to \mathcal J_k$ is a quotient under the twisted identifications $\{\tau_i\}^k_{i=1}$, the last step of this strategy becomes very complicated. Instead, we wish to construct a deformation/specialization of $\widetilde{J_k}$ along with the pairs of divisors and identifications such that after deformation, the identifications between the resulting divisors become untwisted, i.e., they commute with the projection maps to $J_0$. 

\

Before discussing the idea of the solution, let us first recall a definition.

\begin{defe} \label{Special}
A \textbf{specialization} of a projective variety $Z$ to another projective variety $Z_0$ is a proper flat family of varieties $\mathcal Z \rightarrow B$, where $B$ is an irreducible variety such that 
\begin{enumerate}
\item $\mathcal Z_{b_1}$ isomorphic to $Z$ for some $b_1\in B$ and 
\item $\mathcal Z_{b_2}$ is isomorphic to $Z_0$ for some $b_2\in B$.
\end{enumerate}
\end{defe}

We distinguish between a specialization and a deformation because a deformation is defined over a discrete valuation ring, and a specialization is defined over a general base. 

\

\textbf{The idea of the solution:} We construct an algebraic specialization 
\begin{equation}\label{special001}
\widetilde{\mathcal J_k}\rightarrow B^o_k
\end{equation}
of the fiber bundle $\widetilde{J}_k\to J_0$ (we do not deform $J_0$) over a suitable neighbourhood $B^o_k$ of the point $(z_1,\dots, z_k)$ in $(X_0)^k$. The specialization of $\widetilde{J}_k$ induces a specialization of its divisors $\{D_i, D'_i\}^k_{i=1}$ and the identifications $\{\tau_i\}^k_{i=1}$. Moreover, we show that on the fiber of $\widetilde{\mathcal J_k}\rightarrow B^o_k$ at the point $(z_1,\dots, z_k)\in B^o_k$ the induced identifications between the divisors become fiberwise i.e., they all commute with the projection maps to $J_0$. Then,  we construct the quotient under these global identifications. This produces a specialization ${\mathcal J_k}\rightarrow B^o_k$ of the variety $\overline {J_k}$ such that the fibre over the point $(x_1, \dots, x_k)\in B^o_k$ is isomorphic to the compactified Jacobian $\overline {J_k}$ of the irreducible nodal curve $X_k$ and the fiber at the point $(z_1,\dots, z_k)\in B^o_k$ is isomorphic to $J_0\times R^k$, where $R$ is the irreducible rational nodal curve with one node. The fiber over $(z_1,\dots, z_k)\in B^o_k$ becomes so simple because the induced identifications on the fiber of $\widetilde{\mathcal J_k}\rightarrow B^o_k$ at this point commute with the projection maps to $J_0$. Moreover, we show that the family ${\mathcal J_k}\rightarrow B^o_k$ is topologically trivial over the base $B^o_k$. As a consequence, it follows that $\overline {J_k}$ is homeomorphic to $J_0\times R^k$. Moreover, the higher direct image sheaves of the constant sheaf $\mathbb Q$ forms a variation of Hodge structures. Therefore, the mixed Hodge numbers of $\overline {J_k}$ are the same as that of $J_0\times R^k$.

\

\textbf{A possible generalisation.} The problem that we have discussed above also figures in the case of compactification of moduli of vector bundles of higher ranks on a nodal curve. There are two compactifications of moduli of vector bundles on a nodal curve, namely the moduli of torsion-free sheaves \cite{13} and the moduli of Gieseker-vector bundles \cite{G} and \cite{NS II}. It might be possible to generalise our strategy to the higher rank case as well.

\

\textbf{Outline of the paper:} Throughout this article, we will assume that all the curves are irreducible and defined over the field of complex numbers. This article is organized as follows.

\

In $\textbf{\S3}$, we recall a few definitions and constructions necessary for further discussion. \textbf{The results in this section should be well-known; we include these results here for the convenience of the reader.} We discuss torsion-free sheaves of rank one on nodal curves and generalised parabolic bundles of rank one, and then we discuss the compactified Jacobian of irreducible nodal curves, their singularities, and the description of the normalization map. We briefly recall the construction of the so-called $\Theta$ bundle on the Jacobian of smooth curves and on the compactified Jacobian of irreducible nodal curves. 

\

In $\textbf{\S4}$, we carry out the construction of the specialization (definition \eqref{Special}) in the case when the nodal curve has only one node. Here, we outline the strategy of constructing the specialization in the one node case in the following steps.

\

\begin{enumerate}
\item{\textbf{Step 1.}} Consider the family of projective bundles over $J_0$ 

\begin{center}
\begin{tikzcd}
\widetilde{\mathcal J_1}\arrow{d}{\widetilde{\mathfrak f}_1}\\
X_0
\end{tikzcd}
\end{center}
where 
\begin{enumerate}
\item $\widetilde{\mathcal J_1}:=\mathbb P(\mathcal P\oplus p_{2}^*\mathcal P_{z_1})$, 
\item $\mathcal P_{z_1}$ denotes the restriction line bundle $\mathcal P$ on $z_1\times J_0$, and 
\item $p_{2}: X_0\times J_0\rightarrow J_0$ is the projection map.
\end{enumerate}

We call the variety $\widetilde{\mathcal J_1}$ the \textbf{total space}.
\

\item{\textbf{Step 2.}} There are two natural Weil-divisors $\mathcal D_1$ and $\mathcal D'_1$ on $\widetilde{\mathcal J_1}:=\mathbb P(\mathcal P\oplus p_{2}^*\mathcal P_{z_1})$ given by the natural quotient line bundles $\mathcal P\oplus p_{2}^*\mathcal P_{z_1}\rightarrow \mathcal P$ and $\mathcal P\oplus p_2^*\mathcal P_{z_1}\rightarrow p_2^*\mathcal P_{z_1}$, respectively. Both $\mathcal D_1$ and $\mathcal D'_1$ are isomorphic to $X_0\times J_0$ via the projection map $\widetilde{\mathcal J_1}\to X_0\times J_0$. If we fix a point $x\in X_0$, and denote by $\widetilde{\mathcal J_{1,x}}, \mathcal D_{1,x}, \mathcal D'_{1,x}$ the fibres over $x\in X_0$ of the composite maps $\widetilde{\mathcal J_{1}}\to X_0\times J_0\to X_0, \mathcal D_1\to X_0\times J_0\to X_0, \mathcal D'_1\to X_0\times J_0\to X_0$ respectively, then notice that $\widetilde{\mathcal J_{1,x}}=\mathbb P(\mathcal P_x\oplus \mathcal P_{z_1})$ and $\mathcal D_{1,x}$ and $\mathcal D'_{1,x}$ correspond to the two natural quotient line bundles, namely $\mathcal P_x\oplus \mathcal P_{z_1}\to \mathcal P_x$ and $\mathcal P_x\oplus \mathcal P_{z_1}\to \mathcal P_{z_1}$ over $J_0$. See subsection \ref{twodiv} and the figure \ref{fig} for details. 

\item{\textbf{Step 3.}} There is a "twisted" isomorphism $\tau_1: \mathcal D_1\rightarrow \mathcal D'_1$. We call it a "twisted" isomorphism because it does not commute with the projection morphism $\widetilde{\mathfrak f}_1: \widetilde{\mathcal J_1}\rightarrow X_0\times J_0$. If $x\in X_0$ and $x\neq z_1$, then the following diagram is not commutative

\begin{equation}
\begin{tikzcd}
D_{1,x}\arrow["\tau_{1,x}"]{rr} \arrow{dr} && D'_{1,x}\arrow{dl}\\
& \{x\}\times J_0
\end{tikzcd}
\end{equation}

because $\tau_1$ maps the fibre over $[L]\in J_0$ to the fibre over $[L':=L\otimes \mathcal O(z_1-x)]\in J_0.$ This also explains why $\tau_{1,z_1}: D_{1,z_1}\to D'_{1,z_1}$ is an "untwisted"-isomorphism, because $L':=L\otimes \mathcal O(z_1-z_1)=L$. In other words, we have the following commutative diagram

\begin{equation}
\begin{tikzcd}
D_{1,z_1}\arrow["\tau_{1,z_1}"]{rr} \arrow{dr} && D'_{1,z_1}\arrow{dl}\\
& \{z_1\}\times J_0
\end{tikzcd}
\end{equation}

See Proposition \ref{Involution1}, figure \ref{fig}, Remarks \ref{Twist} and \ref{twist2024} and also the proof of the second statement of Theorem \ref{D2024} for details. 

\item{\textbf{Step 4.}}  Using a push-out construction by Artin, we construct a family of algebraic spaces 
\begin{center}
\begin{tikzcd}
\mathcal J_1\arrow{d}{\mathfrak f_1}\\
X_0
\end{tikzcd}
\end{center}
where $\mathcal J_1$ is the algebraic space constructed as a quotient of $\widetilde{\mathcal J_1}$ by identifying the divisors $\mathcal D_1$ and $\mathcal D'_1$ using the twisted isomorphism $\tau_1: \mathcal D_1\to \mathcal D'_1$. Notice since $\tau_{1,z_1}: \mathcal D_{1,z_1}\to \mathcal D'_{1,z_!}$ is an untwisted isomorphism, the fibre of the map $\mathfrak f_1: \mathcal J_1\to X_0$ over the point $z_1\in X_0$ is isomorphic to $\frac{\mathbb P(\mathcal P_{z_1}\oplus \mathcal P_{z_1})}{~\tau_{1,z_1}}\cong J_0\times \mathbb P^1$. This is why we get an untwisted specialisation of $\overline{J_k}$ at the point $z_1\in X_0$. See figure \ref{fig} and also the proof of the second statement of Theorem \ref{Descent} for details. 

\item{\textbf{Step 6.}} We show that $\mathcal J_1$ has normal crossing singularities. 

\item{\textbf{Step 7.}} We fix a point $p_0$ different from $x_1$ and $z_1$. Using the choice of the point, we define a line bundle $\widetilde{\Theta_1}$ on $\widetilde{\mathcal J_1}$ which is relatively ample for the map $\widetilde{\mathfrak f}_1: \widetilde{\mathcal J_1}\rightarrow X_0$.
\item{\textbf{Step 8.}} Here we define and henceforth work with a new base $B^o_1:=X_0\setminus \{p_0\}$ instead of $X_0$. We show that the line bundle $\widetilde{\Theta_1}$ descends to $\mathcal J_1|_{_{B^o_1}}$. Therefore, it follows that the morphism $\mathfrak f_1: \mathcal J_1|_{_{B^o_1}}\rightarrow B^o_1$ is projective. This is the desired family of specialisations in one-node case.
\end{enumerate}
\begin{figure}\label{fig}
\centering

\begin{tikzpicture}

\draw[red,thick,dashed] (-2,11) rectangle (14, 17);
\draw (-1.5,10.6) node{\textcolor{red}{\tiny $\widetilde{\mathcal J_1}$}};

\draw (0,12) rectangle (4,16); \draw (8,12) rectangle (12,16);
\draw (0, 12.5) -- (4,12.5); \draw (0, 15.5) -- (4,15.5); \draw (.5, 12) -- (.5,16); \draw (3.5, 12) -- (3.5,16);
\draw (8, 12.5) -- (12,12.5); \draw (8, 15.5) -- (12, 15.5); \draw (8.5, 12) -- (8.5,16); \draw (11.5, 12) -- (11.5,16);

\draw (4,6) node{.};

\draw [-stealth](6,10) -- (6,9);
\draw[red, -stealth] (8.6,15.3) .. controls (9.9,13.7) and (9.7,13.7) .. (8.6,12.7);
\draw[red, -stealth] (11.6,15.3) .. controls (12.6,13.7) and (12.6,13.7) .. (11.6,12.7);

\draw (-.5, 15.5) node{\tiny $D_{1, x_1}$};
\draw (-.5, 12.5) node{\tiny $D'_{1, x_1}$};
\draw [-stealth, blue, dashed](-.9,13) -- (.4,13); \draw (-2.1,13) node{\tiny $\mathbb P(L_{x_1}\oplus L_{z_1})$};
\draw [-stealth, blue, dashed](-.9,15) -- (3.5,15); \draw (-2.1,15) node{\tiny $\mathbb P(L'_{x_1}\oplus L'_{z_1})$};
\draw (4,13.9) node{\textcolor{red}{\tiny ${[L\mapsto L':=L\otimes \mathcal O(x_1-z_1)]}$}};
\draw (1.5,13.9) node{\textcolor{red}{\tiny ${\tau}$}};

\draw (.9,5.5) node{\tiny $L$}; \draw (3.6,5.5) node{\tiny $L':=L\otimes \mathcal O(x_1-z_1)$};

\draw (2.1,4.5) node{\tiny $x_1\times J_0$}; \draw (10.1,4.5) node{\tiny $z_1\times J_0$}; 

\draw (8.9,5.5) node{\tiny $L=L\otimes \mathcal O(z_1-z_1)$}; 
\foreach \Point in {(.5,6), (3.5,6), (8.5,6)}{
    \node at \Point {\textbullet};
}

\draw [red, -stealth](.6,15.3) -- (3.3,12.7); 
\draw [-stealth](6,10) -- (6,9);

\draw [red, -stealth, dashed](-.3,17.3) -- (.4,15.7); \draw (-.3, 17.6) node{\tiny $(L, L_{x_1}\oplus L_{z_1}\rightarrow L_{x_1})$};

\draw [red, -stealth, dashed](7.7,17.3) -- (8.4,15.7); \draw (7.7, 17.6) node{\tiny $(L, L_{z_1}\oplus L_{z_1}\xrightarrow{1^{st} proj} L_{z_1})$};

\draw [blue, -stealth, dashed](-.3,17.3) -- (.4,15.7); \draw (-.3, 17.6) node{\tiny $(L, L_{x_1}\oplus L_{z_1}\rightarrow L_{x_1})$};

\draw [blue, -stealth, dashed](7.7,17.3) -- (8.4,15.7); \draw (7.7, 17.6) node{\tiny $(L, L_{z_1}\oplus L_{z_1}\xrightarrow{1^{st} proj} L_{z_1})$};

\draw [-stealth, blue, dashed](7.1,15) -- (8.5,15); \draw (6,15) node{\tiny $\mathbb P(L_{z_1}\oplus L_{z_1})$};

\draw (9,13.9) node{\textcolor{red}{\tiny ${[L\mapsto L':=L\otimes \mathcal O(x_1-z_1)]}$}};

\draw (9,14.5) node{\textcolor{red}{\tiny ${\tau}$}};

 \draw (10.1,11.5) node{\tiny $\widetilde{\mathcal J}_{1,z_1}$}; 
 \draw (2.1,11.5) node{\tiny $\widetilde{\mathcal J}_{1,x_1}$}; 

\draw [blue, -stealth, dashed](7.7,10.7) -- (8.4,12.3); \draw (7.9, 10.4) node{\tiny $(L, L_{z_1}\oplus L_{z_1}\xrightarrow{2^{nd} proj} L_{z_1})$};

\draw [blue, -stealth, dashed](2.7,10.7) -- (3.4,12.3); \draw (2.6,10.5) node{\tiny $(L', L'_{x_1}\oplus L'_{z_1}\xrightarrow{} L'_{z_1})$};

\draw (7.5, 15.5) node{\tiny $D_{1, z_1}$};
\draw (7.5, 12.5) node{\tiny $D'_{1, z_1}$};
\draw (6,6) node{$\dots \dots$};
\draw (6,13) node{$\dots \dots$};

\draw[red,thick,dashed] (-2,4) rectangle (14, 8);

\draw (-1.5,3.6) node{\textcolor{red}{\tiny $X_0\times J_0$}};

\draw (2,6) ellipse (2.5cm and 1cm);   \draw (10,6) ellipse (2.5cm and 1cm);

\draw [-stealth](6,3) -- (6,2);

\end{tikzpicture}

\includegraphics[scale=.4]{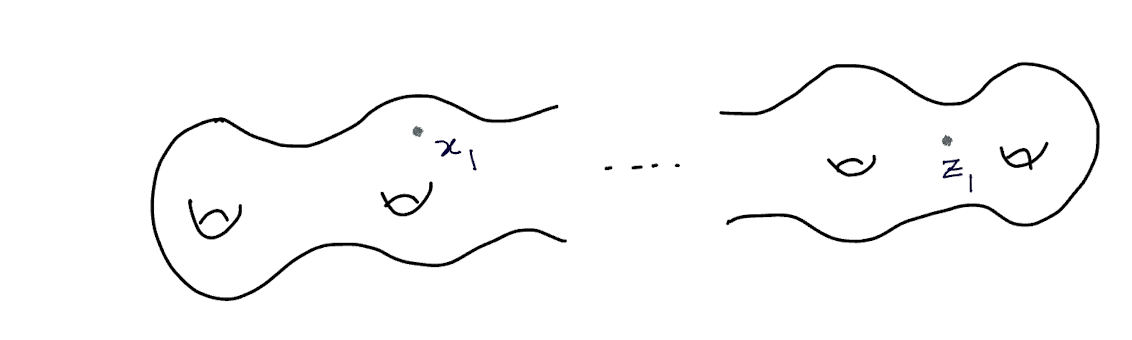}

 \caption{one node case}
\end{figure}

\

In $\textbf{\S5}$, we generalise this construction for any irreducible nodal curve with $k$ nodes for any positive integer $k$. The construction is very similar to the construction in the one-node case. In this case, we start with the following variety as the \textbf{total space}
\begin{center}
\begin{tikzcd}
\widetilde{\mathcal J_k}\arrow{d}{\widetilde{\mathfrak f}_k}\\
X_0^k
\end{tikzcd} 
\end{center}
where
\begin{enumerate}
\item $\widetilde{\mathcal J}_k:=\mathbb P(p_{1, k+1}^*\mathcal P\oplus p_{k+1}^* \mathcal P_{z_1})\times_{_{X_0^k\times J_0}}\cdots\times_{_{X_0^k\times J_0}} \mathbb P(p_{k,k+1}^*\mathcal P\oplus p_{k+1}^* \mathcal P_{z_k})$,
\item $p_{i,k+1}: X_0^k\times J_0\rightarrow X_0\times J_0$ denotes the projection to the product of the $i$-th copy of $X_0$ and $J_0$.
\item $p_{k+1}$ denotes the projection $X_0^k\times J_0\rightarrow J_0$.
\end{enumerate}

 There are $k$ pairs of divisors $\{\mathcal{D}_i, \mathcal{D}'_i\}^k_{_{i=1}}$, where $\mathcal D_i$ and $\mathcal D'_i$ are the two natural divisors pulled back from $\mathbb P(p_{i, k+1}^*\mathcal P\oplus p_{k+1}^* \mathcal P_{z_i})$. There are  $k$ natural "twisted" isomorphisms $\tau_i:\mathcal{D}_i\rightarrow \mathcal{D}'_i$  (see Lemma \ref{LemaInvo} for details). Intuitively, $\tau_i$ is the pullback of the twisted isomorphism between the two natural divisors in $\mathbb P(p_{i, k+1}^*\mathcal P\oplus p_{k+1}^* \mathcal P_{z_i})$. Unlike the single node case though, here the isomorphisms $\tau_i$ are not defined everywhere. The isomorphisms exist only when we focus on the following new base 
$$
B_k:=\{(b_1,b_2,\dots,b_k)\in X_0^k ~~|~~b_i\neq b_j ~~\text{and}~~b_i\neq z_j~~ \text{ for }1\leq i,j\leq k \text{ and }i\neq j \}.
$$
We therefore restrict our attention over $B_k$.

We construct a quotient space $\mathcal{J}_k$ over $B_k$, inductively, as a quotient of $\widetilde{\mathcal J_k}$ by identifying the divisors in every pair using the twisted isomorphisms between them. Repeated application of proposition \cite[Theorem 3.1]{1} and \cite[Theorem 45 (Gluing of algebraic spaces).]{K} shows that the quotient space is an algebraic space. We denote it by $\mathfrak f_k: \mathcal J_k\to B_k$.

We further show that the algebraic space $\mathcal J_k$ has product of normal crossing singularities. More precisely, the analytic local ring at a singular point is formally smooth to $\frac{\mathbb C[u_1,v_1,\cdots, u_k,v_k]}{u_1\cdot v_1,\dots, u_i\cdot v_i}$ for some $1\leq i\leq k$.
 We choose and fix a point $p_0\in X_0\setminus \{x_1,\dots, x_k, z_1,\dots, z_k\}$. With this choice of point, we define a line bundle $\widetilde{\Theta}_k$ on $\widetilde{\mathcal J}_k$ which is relatively ample for the morphism $\widetilde{\mathfrak f}_k: \widetilde{\mathcal J}_k\rightarrow B_k$. We refer to it as the Theta bundle on $\widetilde{\mathcal J}_k$. But the morphism ${\mathfrak f}_k: \mathcal J_k\to B_k$ may not be a projective morphism. To get projectivity we need to choose a further smaller open set of $X_0^k$ and focus on it. The new base is the following
 
 $$
 B^o_k:=\{(x_1,\dots,x_k)\in B_k~~|~~x_i\neq p_0~~\text{for all}~~1\leq i\leq k\}. 
 $$
 
 We show that over this new base the line bundle $\widetilde{\Theta}_k$ descends to $\mathcal J_k|_{_{B^o_k}}$. This implies that the morphism $\mathfrak f_k: \mathcal J_k|_{_{B^o_k}}\rightarrow B^o_k$ is projective. This is the desired family of specialisations in the multinode case. We summarise the content of \textbf{\S5} in the following theorems.

\

\begin{thm}
The quotient space $\mathcal{J}_k$ is an algebraic space and has the product of normal crossing singularities.
\end{thm} 

\

\begin{thm}\begin{enumerate}

\item The line bundle $\widetilde{\Theta_k}$ on $\widetilde{\mathcal J_k}$ descends to $\mathcal{J}_k$. In other words, $\mathfrak f_k: \mathcal{J}_k\rightarrow B^o_k$ is a projective morphism.
\item Let $X_k$ be a nodal curve with nodes at $y_i$ and $q_k:X_0\rightarrow X_k$ be the normalization with $q_k^{-1}(y_i)=\{x_i,z_i\}$. Then
\begin{align*}
\mathfrak f_k^{-1}(b_1,\dots, b_k)=&\overline{J}_{X(b_1,\dots,b_k)}\hspace{50pt}\text{when} ~~(b_1,\cdots, b_k)\in B^o_k ~~\text{and}~~ b_i\neq z_i~ \forall i =1,\dots, k,\\
=&J_0\times R^k\hspace{60pt}\text{when}~~b_i=z_i~\forall i=1,\cdots, k,
\end{align*}
\end{enumerate}
where $X(b_1,\dots,b_k)$ denotes the nodal curve obtained from $X_0$ by identifying $b_i$ with $z_i$ for every $i=1, \dots, k$ and $\overline J_{X(b_1,\dots, b_k)}$ denotes its compactified Jacobian. 
\end{thm}

\

In $\textbf{\S6}$, we study the natural stratification of the variety $\mathcal J_k$ given by its successive singular locus. We show that this stratification satisfies the Whitney's conditions. Then, by using Thom-Mather's first isotropy lemma, we conclude the following.

\

\begin{thm}\label{GG}
\begin{enumerate}
\item The morphism $ \mathfrak f_k:\mathcal J_k\rightarrow B^o_k$ is topologically locally trivial.
\item $\mathcal{R}^i \mathfrak f_{k*}\mathbb Q$ forms a \text{variation of mixed Hodge structures} over $B^o_k$.
\end{enumerate}
\end{thm}

\

In \textbf{\S7}, we discuss some applications of the construction of the specialization. As a corollary of the above theorem \ref{GG}, we see that the Betti numbers of $\overline J_k$ and the mixed Hodge numbers of the cohomology groups of $\overline{J}_k$ are the same as the Betti numbers of $J_0\times R^k$ and the mixed Hodge numbers of the cohomology groups of $J_0\times R^k$, respectively. We compute the Betti numbers and mixed Hodge numbers of the cohomology groups of the latter using the Kunneth formula.

\

\begin{thm}

\begin{enumerate}

\item Then $i$-th betti number of $\bar{J}_k$
\begin{equation}\label{we123}
h^{i}(\bar{J}_k)=h^{i}\left(J_0\times R^k\right)=\sum_{0\leq l\leq \tt{min}\{i,2k\}}\binom{2(g-k)}{i-l}.\sum_{\frac{1}{2}\leq j\leq \tt{min}\{l,k\}}\binom kj. \binom j{2j-l}.
\end{equation}

\item The dimension of $gr_{l}^{W}\left(H^{i}(\bar{J}_k)\right)$ is
\begin{equation}\label{we1234}
\dim_{\mathbb{Q}}\gr_{l}^{W}\left(H^{i}(\bar{J}_k)\right)= \sum_{0\leq t\leq l, (l-t)~ \text{is even}} \binom {2(g-k)}{t}.\binom {k}{i-\frac{l-t}{2}}. \binom{i-\frac{l-t}{2}}{i-l+t} 
\end{equation}
and
\item
For $p$, $q\geq0$ such $p+q=l$, the dimension of 
\[\dim_{\mathbb{C}}\gr_F^p\gr^p_{\bar{F}}\left(\gr_{l}^{W}\left(H^{i}(\bar{J}_k)\right)\right)=\sum_{0\leq t\leq l, (l-t)~ \text{is even}} \binom {g-k}{p-\frac{l-t}{2}}\binom{g-k}{q-\frac{l-t}{2}}\binom {k}{i-\frac{l-t}{2}}. \binom{i-\frac{l-t}{2}}{i-l+t}\]
 \end{enumerate}
\end{thm}

\

\subsection{Acknowledgement} We are grateful to the referee for the useful comments and suggesting very interesting references related to this work. 

\newpage
\begin{table}[h]\label{Not1}
    \caption{Notation and Convention}
    \begin{tabularx}{\textwidth}{p{0.40\textwidth}X}
   \\ \toprule
      \\ 
      $\mathbb C$ & The field of complex numbers.   \\ \\
      
      $g$ & arithmetic genus of the nodal curve, under study.\\ \\
      
      & We will work with $\mathbb C$ as our base field. \\ \\
      
    $\binom nr$  & $\frac{n!}{r!\cdot (n-r)!}$, $\binom rn:=0  ~~\text{for}~~r<n$ \\ \\
    
     $p_{i_1\cdots i_l}: Z_1\times\cdots\times Z_n\rightarrow \prod^{j=l}_{j=1} Z_{i_j}$ & Let $n$ be an integer and $Z_1,\dots, Z_n$ are $n$ varieties. For any ordered subset ${i_1<\cdots<i_l}$ of $\{1,\dots, n \}$, we denote by $p_{i_1\cdots i_l}$ the obvious projection morphism $Z_1\times\cdots\times Z_n\rightarrow \prod^{j=l}_{j=1} Z_{i_j}$. \\  \\
     
     $R$ & the rational nodal curve with a single node. \\ \\
     $X_0$ & a smooth projective curve of genus $g$ \\  \\
     
     $X_k$ & an irreducible nodal curve of arithmetic genus $g$ with $k$ nodes \\ \\
     
     $X_0^k$ & $\underbrace{X_0\times \cdots \times X_0}_{k~~\text{times}}$ \\ \\

          $q_k: X_0\to X_k$ &  The normalization map \\  \\
          
           $X(x_1,\dots,x_k)$ & Fix $k$ different points $\{z_1,\dots,z_k\}$ of $X_0$. Then for any $\{x_1,\dots, x_k\}\in (X_0\setminus \{z_1,\dots, z_k\})$, we denote by $X(x_1,\dots,x_k)$ the irreducible nodal curve $X(x_1,\dots,x_k)$ constructed as a quotient of $X_0$ by identifying $x_i$ with $z_i$ for every $1\leq i\leq k$. \\  \\

      \bottomrule
     \end{tabularx}
\end{table}

\newpage

\begin{table}[h]
    \caption{Notation and Convention}
    \begin{tabularx}{\textwidth}{p{0.40\textwidth}X}
   \\ \toprule
      \\ 
                  $J_0$ & the Jacobian of the curve $X_0$ \\ \\
		$\overline{J_k}$ & the compactified Jacobian of $X_k$ \\ \\
		$\mathcal P$ &  a Poincar\'e line bundle on $X_0\times J_0$ or a Poincar\'e sheaf on $X_k\times \overline{J_k}$ (see def. \ref{Poincare101} and \ref{Poincare102}) \\ \\
		$\nu_k: \widetilde{J_k}\to \overline{J_k} $ & the normalization map (see diag. \ref{diag121}) \\ \\
		$Det~\mathcal F$ & the determinant of cohomology of $\mathcal F$ (see def. \ref{detcoh}) \\ \\
		$\widetilde{\mathcal{J}_k}$ & the total space \eqref{Def101}\\ \\
		$\mathcal{J}_k$ & the specialization in the case of $k$ nodes (see definition \ref{special1102})\\ \\
		$(\mathcal D_i, \mathcal D'_i)$ & the natural pair of divisors on $\widetilde{\mathcal{J}_k}$ for every $i=1, \dots, k$ (see subsection \ref{pairdiv}) \\ \\
		$\tau_i:\mathcal D_i\to \mathcal D'_i$ & the twisted isomorphism for $i=1, \dots, k$ (see Lem. \ref{LemaInvo})\\ \\
		$\widetilde{\Theta_k}$ & the theta bundle on $\widetilde{\mathcal J_k}$ (see \ref{Theta1103})\\ \\
		$\Theta_k$ & the theta bundle on $\mathcal J_k$(see proof of Thm. \ref{MT11})\\ \\
		$\mathfrak f_k: {\mathcal J_k}\to B^o_k$ & the specialization in the case of $k$ nodes (see Thm. \ref{MT11})\\ \\
		$B^o_k$ & the base of the specialization in the case of $k$ nodes (see def. \ref{special1102})\\ \\
		
           \bottomrule
     \end{tabularx}
\end{table}

\newpage 

\section{\textbf{Preliminaries}}

The results in this section should be well-known; we include them here for reader's convenience.

\subsection{\textbf{Torsion-free sheaves of rank one and generalized parabolic bundles}}
Let $X_k$ be an irreducible projective nodal curve of arithmetic genus $g$ with exactly $k$ nodes
$\{y_1,\dots,y_k\}$. Let us denote by $q_k: X_0\rightarrow X_k$ the normalization of $X_k$. Let $\{x_i, z_i\}$ denote the inverse images of the node
$y_i$ for each $i=1,\dots, k$. Note that the genus of $X_0$ is $(g-k)$. 

\

Let $\mathcal F$ be a torsion-free sheaf of rank $1$ over $X_k$, which is not locally free at a node $y_i$. Let ${\mathcal F}_{(y_i)}$ denote the localisation of $\mathcal F$ at the node $y_i$. Then 

$${\mathcal F}_{(y_i)}\cong m_{(y_i)}$$

where $m_{(y_i)}$ denotes the maximal ideal in the local ring at the node $y_i$ (\cite[Proposition 2, Page 164]{13}). 

\

\begin{defe}
A \textbf{generalised parabolic bundle (GPB)} of rank one over $X_0$ is a $(k+1)$-tuple $(E, Q_1,\dots, Q_k)$, where $E$ is a line bundle on $X_0$ and $E_{x_i}\oplus E_{z_i}\rightarrow Q_i$ is a quotient of dimension $1$ for each $i=1,\dots, k$. By abuse of notation, we will also denote the quotient maps by $Q_i$. 

\

The \textbf{degree of a GPB} $(E, Q_1,\dots, Q_k)$ is defined to be the degree of the underlying bundle $E$. 
\end{defe}

\

\subsubsection{\textbf{Torsion-free sheaf corresponding to a GPB}} Given a GPB $(E, Q_1,\dots, Q_k)$ of rank $1$ there is the following canonical rank $1$ torsion-free sheaf induced by the GPB.

\begin{equation}\label{Induce1}
\mathcal F:=Kernel((q_k)_*E\rightarrow \oplus_{i=1}^{k} Q_i),
\end{equation}

where $q_k:X_0\to X_k$ is the normalization map. We will refer to $\mathcal F$ as the torsion-free sheaf induced by the GPB $(E, Q_1,\dots, Q_k)$.  

\

\begin{prop} \label{GPBs}
Let $\mathcal F$ be a torsion-free sheaf of rank one and degree $d$ on $X_k$, which is not locally free exactly at the nodes $\{y_1,\dots, y_r\}$. Then there are exactly $2^r$ different GPBs of degree $d$ which induce the same torsion-free sheaf $\mathcal F$.
\end{prop}
\begin{proof}
Consider the line bundle $E':=\frac{q_k^*\mathcal F}{\tt{Torsion}}$ of degree $=(\text{deg}~\mathcal F-r$). There are $2^r$ GPBs of rank one 
$$(E, Q_i, \forall i=1,\dots, k),$$ 
where 
\begin{enumerate}
\item $E:=E'\otimes \mathcal O(p_1+\cdots+p_r)$, 
\item $p_i\in \{x_i,z_i\}$, and
\item $Q_i$ is the quotient $E_{x_i}\oplus E_{z_i}\rightarrow E_{p_i}$ for $i=1,\dots, r$, and
\item for $i=r+1,\dots, k$, the quotient $Q_i$ is 

$$E_{x_i}\oplus E_{z_i}\rightarrow \frac{E_{x_i}\oplus E_{z_i}}{\Gamma_{\phi_i}},$$ 
where $\Gamma_{\phi_i}$ is the graph of the natural isomorphism $\phi_i:E_{x_i}\rightarrow E_{z_i}$ induced by $\mathcal F$. 
\end{enumerate}

Notice that any tuple $(E, Q_i, \forall i=1,\dots, k)$ as above determines a short exact sequence  (as in \eqref{Induce1})  
\begin{equation}\label{Incloo}
0\rightarrow\Ker(\gamma)\rightarrow (q_k)_*E\overset{\gamma}{\rightarrow} \oplus_{i=1}^{k} Q_i\rightarrow 0.
\end{equation}
By a local calculation (\cite[Lemma 6.1]{BD}), one can see that there is a natural inclusion $\mathcal{F}\rightarrow \Ker(\gamma)$ of sheaves and that the inclusion is, in fact, an isomorphism. From \cite[Lemma 6.1]{BD} it also follows that deg $E$=deg $\mathcal F$. Using \eqref{Incloo}, we easily see that these are the GPBs such that the induced torsion-free sheaves are isomorphic to $\mathcal F$.
\end{proof}

\

\begin{rema}\label{Twist}
Let $(E:=E'\otimes \mathcal O(p_1+\dots+p_r), \{Q_i\}_{i=1}^k)$ and $(F:=E'\otimes \mathcal O(p'_1+\dots +p'_r), \{Q'_i:=F_{p'_i}\}_{i=1}^k)$ are two such GPBs (as above) of rank $1$ over $X_0$ which induce the same torsion-free sheaf $\mathcal F$. Here $p_i, p'_i\in \{x_i,z_i\}$ for $i=1,\dots, r$. Then, from the proof of the previous proposition, it is clear that the underlying bundles of the GPBs are related by the following "twist" or so-called " Hecke modification." 
\begin{equation}
E\mapsto E\otimes \mathcal O(\sum_{i=1}^r (p'_i-p_i))\cong F
\end{equation}

For more details, see subsubsection \ref{singloc} and Remark \ref{twist1119}
\end{rema}

\

\subsection{\textbf{Compactified Jacobian and its normalization}}
There exists a projective variety, which parametrizes all the torsion-free sheaves over $X_k$ of rank $1$ and degree $0$. It is known as the Compactified Jacobian of the curve $X_k$ \cite{11}. Let us denote it by $\overline{J}_{_k}$. Let us denote by $J_0$ the Jacobian of the curve $X_0$. 

There exists a line bundle $\mathcal P$ over $X_0\times J_0$ such that for any point $[E]\in J_0$ the restriction of the line bundle $\mathcal P|_{_{(X_0\times [E])}}$ is isomorphic to $E$ over $X_0$. A line bundle with this property is called a Poincar\'e line bundle (\ref{Poincare101}). We choose and fix one such line bundle $\mathcal P$.

\

For each $i=1, \dots, k$, we have the following $\mathbb P^1$- bundle $\mathbb P(\mathcal P_{x_i}\oplus \mathcal P_{z_i})$ over $J_0$. We consider the fiber product $\mathbb P(\mathcal P_{x_1}\oplus \mathcal P_{z_1})\times_{J_0}\cdots\times_{J_0} \mathbb P(\mathcal P_{x_k}\oplus \mathcal P_{z_k})$ over $J_0$. Notice that by definition, the fibre product is the space that parametrizes all GPBs over $X_0$ of rank $1$ and degree $0$, which induce torsion-free sheaves of rank $1$ and degree $0$ on the nodal curve $X_k$. Let us denote it by $\widetilde{J_k}:=\mathbb P(\mathcal P_{x_1}\oplus \mathcal P_{z_1})\times_{J_0}\cdots\times_{J_0} \mathbb P(\mathcal P_{x_k}\oplus \mathcal P_{z_k})$.

\

We get the following diagram:
\begin{equation}\label{diag121}
\begin{tikzcd}
& \widetilde{J_k}\arrow{dl}{\nu_k}\arrow{dd} \\
\overline{J}_{X_k}\\
& J_0
\end{tikzcd}
\end{equation}

where $\nu_k: \widetilde{J_k}\to \overline{J}_{X_k}$ is the map given by $(E, Q_1,\dots, Q_k)\to \mathcal F:=Kernel((q_k)_*E\rightarrow \oplus_{i=1}^{k} Q_i)$ as in \eqref{Induce1}.

From proposition \ref{GPBs}, it follows that the map $\nu_k$ is a finite, birational morphism. Also notice that since $\widetilde{J_k}$ is proper and smooth, it must be the normalization of $\overline{J}_{k}$.

\subsubsection{\textbf{Singular loci and twisted isomorphisms}}\label{singloc} Consider a torsion free sheaf $\mathcal F$ of rank $1$ and degree $0$ over $X_k$, which is not locally free at exactly $r$ nodes. From \cite[Remark, page 62]{11}, it follows that the complete local ring of the variety $\overline{J}_{k}$ at the point $[\mathcal F]$  is formally smooth to the ring $\frac{k[|t_1,\dots, t_{2r}|]}{(t_1t_2,\dots, t_{2r-1}t_{2r})}$. Therefore, from proposition \ref{GPBs}, it follows that the morphism $\widetilde{J_k}\rightarrow \overline{J}_k$ is an isomorphism over the locus consisting of locally free sheaves of rank one on $X_k$. Consider the divisor $S_1$ consisting of points $\mathcal F$, which are not locally free at some of the nodes of the curve $X_k$. From the descriptions of the complete local rings, it also follows that this divisor is precisely the singular loci of $\overline {J_k}$.

\

Consider the locally closed subset $S^0_1$ of $\overline{J}_k$ consisting of torsion-free sheaves that are not locally free exactly at one node. Then clearly, $S^0_1$ is the disjoint union $\coprod_{i=1}^k S^0_{1, i}$, where $S^0_{1, i}$ is the locus consisting of torsion-free sheaves which are not locally free exactly at $y_i$. It follows from proposition \ref{GPBs} that the inverse image in $\widetilde{J_k}$ of any point $[\mathcal{F}]\in S^0_{1, i}$ consists of exactly two points. These two points can be described as the following two GPBs using the proof of proposition \ref{GPBs}.
\begin{enumerate}
\item $(E:=E'\otimes \mathcal O(x_i), \{E_{x_j}\oplus E_{z_j}\to Q_j\}^k_{j=1})$
\item $(F:=E'\otimes \mathcal O(z_i), \{F_{x_j}\oplus F_{z_j}\to Q'_j\}^k_{j=1})$
\end{enumerate}

where
\begin{enumerate}
\item $Q_i:=E_{x_i}$;
\item for $j\neq i$, the quotient $Q_j$ is 

$$E_{x_j}\oplus E_{z_j}\rightarrow \frac{E_{x_j}\oplus E_{z_j}}{\Gamma_{\phi_j}},$$ 
where $\Gamma_{\phi_j}$ is the graph of the natural isomorphism $\phi_j:E_{x_j}\rightarrow E_{z_j}$ induced by $\mathcal F$. 

\item $Q'_i:=F_{z_i}$
\item $Q'_j=Q_j$ for $j\neq i$.

\end{enumerate}

\begin{rema}{\textbf{(Twisted isomorphisms)}} \label{twist2024}The two GPBs, described above, are related by an isomorphism (remark \ref{Twist}), which we call "twisted isomorphism". It can be described as follows. 

\begin{equation}\label{twist1012}
(E, \{Q_j\}_{j=1}^k)\mapsto (F:=E\otimes \mathcal O(z_i-x_i), \{Q'_j\}_{j=1}^k)
\end{equation}

where 
\begin{enumerate}
\item $Q_i$ is the quotient map $E_{x_i}\oplus E_{z_i}\rightarrow E_{x_i}$, and 
\item for each $j\neq i$, $Q_j$ is a $1$-dimensional quotient of $E_{x_i}\oplus E_{z_i}$, different from $E_{x_j}$ and  $E_{z_j}$, and 
\item $Q'_i$ is the quotient map $F_{x_i}\oplus F_{z_i}\rightarrow F_{z_i}$, and 
\item $Q'_j=Q_j$ for all $j\neq i$.
\end{enumerate} 

Notice that the isomorphism does not commute with the projection map to the Jacobian $J_0$ because the underlying line bundle is twisted by $[E]\mapsto [E\otimes \mathcal O(z_i-x_i)]$ .
\end{rema}

\begin{rema} \label{twist1119} Since $Q'_j$ and $Q_j$ are quotients of the fibers of two different line bundles $F$ and $E$, respectively, the equality in (4) needs a justification, which is as follows. 

Consider $U:=X_0\setminus\{x_i,z_i\}$. The line bundle $F$ is a tensor product of $E$ and a degree zero line bundle $\OO(x_i-z_i)$. Notice that any non-zero constant function on $U$ defines a section of $\OO(x_i-z_i)$ on $U$, by definition. In other words,  $$k_U\subset H^0(U,\OO(x_i-z_i)),$$ where $k_U$ is the set of all constant sections on $U$. Let $\lambda_U$ denote the section with a constant non-zero value $\lambda$ over $U$. 

For every $j\neq i$, and $x_j$, $z_j\in U$, one has the following isomorphism between the fibers of $E$ and $F$ at the  point $x_j$.  
\[E_{x_j}\rightarrow F_{x_j}\cong E_{x_j}\otimes \OO(x_i-z_i)_{x_j}\]
\[\sigma\mapsto\sigma\otimes\lambda_U,\]
 
Similarly, one obtains the following identification of the fibers at $z_j$. For  $\sigma'\in E_{z_j}$, one has
\[E_{z_j}\rightarrow F_{z_j}\cong E_{z_j}\otimes \OO(x_i-z_i)_{z_j}\]
\[\sigma'\mapsto\sigma'\otimes\lambda_U\]
Hence, we obtain natural identifications $E_{x_j}\cong F_{x_j}$ and $E_{z_j}\cong F_{z_j}$ when $j\neq i$. This induces an identification  between $\mathbb{P}(E_{x_j}\oplus E_{z_j})$ and $\mathbb{P}(F_{x_j}\oplus F_{z_j})$. Notice that the identification does not depend on the choice of $\lambda$ since the same $\lambda$ has been used in both identifications. In other words, $Q'_j:=Q_j$ for all $j\neq i$. 
\end{rema}

\

\begin{rema}
More generally, for any subset $I\subset \{y_1, \dots, y_k\}$, we define the following locally closed subset of $\overline{J_k}$
$$
S^0_{I}:=\{\mathcal F\in \overline{J_k}: \mathcal F ~\text{is not locally free exactly at the nodes}~I\}
$$
From proposition \ref{GPBs}, it follows that the preimage under the map $\widetilde{J_k}\to \overline{J_k}$ of a point $[\mathcal F]\in S^0_I$ consists of exactly $2^{\#I}$ distinct GPBs. Moreover, these GPB's are related to each other by twisted isomorphisms. We will discuss these isomorphisms in detail in section $3$ and $4$.
\end{rema}

\subsection{\textbf{Determinant of cohomology and the Theta divisor}}
In this subsection, we will recall the construction of a line bundle called determinant of cohomology over the Jacobian or the compactified Jacobian of a curve. We will also recall the fact that these line bundles coincide with the line bundles corresponding to the so-called "Theta divisors". We need the following theorem.

\begin{thm}\cite[Theorem, Sub-chapter 5]{10} \label{DET1101}
Let $f: X\rightarrow Y$ be a proper morphism of Noetherian schemes with $Y=$Spec $A$ affine. Let $\mathcal F$ be a coherent sheaf on $X$, flat over $Y$. There is a finite complex $K^{\bullet}: 0\rightarrow K^0\rightarrow K^1\rightarrow \cdots K^n\rightarrow 0$ of finitely generated projective $A$-modules and an isomorphism of functors

\begin{equation}
H^p(X\times_Y \text{Spec}~B, \mathcal F\otimes_A B)\cong H^p(K^{\bullet}\otimes_A B), (p\geq 0)
\end{equation}
on the category of $A$-algebras $B$.
\end{thm}

\

The \textbf{determinant of cohomology of $\mathcal F$} is defined as $\text{Det}~\mathcal F:=\otimes_{i=0}^{n} ~~(\text{det}~~K^i)^{(-1)^{i-1}}$ over $Y$.

\

\begin{rema}
The line bundle Det $\mathcal F$ does not depend on the choice of such a finite complex. Therefore, the construction of the line bundle $\text{Det}~\mathcal F$ also holds over arbitrary base $Y$. Also, if the sheaves $R^if_* \mathcal F$ are locally free then Det $\mathcal F\cong \otimes_{i=0}^{n} ~(\text{det}~R^if_*\mathcal F)^{(-1)^{i-1}}$. For further details, we refer to \cite[Chapter VI, pages 134-135]{8}. The determinant of cohomology defined in this article is the inverse of the determinant of cohomology as defined in \cite{8}.
\end{rema}

\

\begin{defe}\label{detcoh}
Let $f: X\rightarrow Y$ be a proper morphism and $\mathcal F$ be a coherent sheaf on $X$, flat over $Y$. We define the \textbf{determinant of cohomology} on $Y$ to be
\begin{equation}
\text{Det}~\mathcal F:=\otimes_{i=0}^{n} ~(\text{det}~K^i)^{(-1)^{i-1}}
\end{equation}
\end{defe}

\

\begin{defe}\label{Poincare101} Let $\mathcal P$ be a line bundle over $X_0\times J_{_{0}}$ with the following properties:
\begin{enumerate}
\item $\mathcal P$ is a flat family of line bundles of degree $0$ on $X_0$ parametrized by $J_0$ ,
\item the morphism given by $[E]\mapsto \mathcal P|_{_{X_0\times [E]}}$ is an isomorphism between $J_0$ and the space of isomorphism classes of line bundles of degree $0$ on $X_0$.
\end{enumerate}

We call such a line bundle a \textbf{Poincar\'e line bundle}, and we denote its determinant of cohomology by $\text{Det}~\mathcal P$.
\end{defe}

\

\begin{defe}\label{Poincare102} Let $\mathcal F$ be a sheaf over $X_k\times \overline J_{_{k}}$ with the following properties:
\begin{enumerate}
\item $\mathcal F$ is flat over $\overline J_k$,
\item $\mathcal F$ is a flat family of rank $1$ torsion-free sheaves of degree $0$ on $X_k$ parametrized by $\overline J_k$,
\item the morphism given by $[F]\mapsto \mathcal F|_{_{X_k\times [F]}}$ is an isomorphism between $\overline J_k$ and the space of isomorphism classes of torsion-free sheaves of rank $1$ and degree $0$ on $X_k$.
\end{enumerate}

We call such a sheaf a \textbf{Poincar\'e sheaf}, and we denote its determinant of cohomology by $\text{Det}~\mathcal F$.
\end{defe}

\

Let us now recall the construction of theta divisors on $J_0$ and $\overline{J_k}$.

\begin{enumerate}
\item{\underline{\textbf{Theta divisor on $J_0$}:}} Fix a point $x_0$ on $X_0$. There is a canonical embedding

\begin{align*}
\phi: X_0\rightarrow J_0\\
x\mapsto \mathcal O_{X_0}(x-x_0)
\end{align*}

The theta divisor on $J_0$ is the schematic image of the map
\begin{equation}
X_0^{g-k-1}\rightarrow J_0
\end{equation}
given by $(x_1,\dots,x_{g-k-1})\mapsto \phi(x_1)\otimes \cdots\otimes \phi(x_{g-k-1}) $. We denote this divisor by $\Theta_0$.

\

\item{\underline{\textbf{Theta divisor on $\overline{J}_{_k}$:}}} Fix a smooth point $x_0$ on $X_k$. There is a canonical embedding
\begin{align*}
\phi_k: X_k\setminus\{y_1,\dots,y_k\}\rightarrow \overline{J}_k\\
x\mapsto I_x^{\vee}\otimes \mathcal O_{X_k}(-x_0)
\end{align*}
where $I_x$ is the ideal sheaf of the point $x$ and $I_x^{\vee}$ denotes the dual sheaf of $I_x$. Then the theta divisor on $\overline{J_k}$ is the schematic closure of the schematic image of the map
\begin{equation}
(X_k\setminus \{y_1,\dots,y_k\})^{g-1}\rightarrow \overline{J}_k
\end{equation}
given by $(x_1, \dots, x_{g-1})\mapsto \phi_k(x_1)\otimes \cdots\otimes \phi_k(x_{g-1})$. We denote this divisor by $\Theta_k$.
\end{enumerate}

\

\begin{prop}\label{Ample1} \cite[Theorem 1. (A)]{NR} For any Poincar\'e line bundle $\mathcal P$ over $X_0\times J_0$ and any Poincar\'e sheaf $\mathcal F$ over $X_k\times \overline J_k$, we have:
\begin{enumerate}
\item $\text{Det} ~\mathcal P\otimes \mathcal P_{x_0}^{\otimes -(g-k-1)}\cong \mathcal O_{J_0}(\Theta_0)$,
\item $\text{Det}~\mathcal F \otimes \mathcal F^{\otimes -(g-1)}_{x_0}\cong \mathcal O_{\overline{J}_k}(\Theta_k)$
\item Let $Q_1,\dots, Q_k$ be the universal quotient bundles over $\widetilde{J}_k$ i.e.,
\begin{equation}
\mathcal P_{x_i}\oplus \mathcal P_{z_i}\rightarrow Q_i
\end{equation}
for $i=1,\dots, k$. Then, $\text{Det} ~\mathcal P\otimes (\otimes_{i=1}^k Q_i)\otimes \mathcal P^{\otimes -(g-1)}_{_{x_0}}\cong ~~\nu_k^*(\text{Det}~\mathcal F\otimes \mathcal F^{\otimes -(g-1)}_{_{x_0}})$.
\item the line bundles $ \mathcal O_{J_0}(\Theta_{0})$ and $\mathcal O_{\overline{J}_k}(\Theta_k)$ are ample line bundles on $J_0$ and $\overline{J}_k$, respectively. Moreover, the line bundle $\text{Det}~\mathcal P\otimes \mathcal P^{\otimes -(g-1)}_{x_0}\otimes (\otimes_{i=1}^k Q_i)$ is an ample line bundle on $\widetilde{J}_k$.
\end{enumerate}
\end{prop}
\begin{proof}
Consider the following short exact sequence of sheaves over $X_0$
\begin{equation}
0\rightarrow \mathcal O_{_{X_0}}\rightarrow \mathcal O_{_{X_0}}(x_0)\rightarrow \mathcal O_{_{X_0}}(x_0)_{_{x_0}}\rightarrow 0
\end{equation}
By pulling back this short exact sequence by the map $p_{_{X_0}}: X_0\times J_0\rightarrow X_0$ and tensoring with $\mathcal P$ we get:
\begin{equation}
0\rightarrow \mathcal P \rightarrow \mathcal P\otimes p_{_{X_0}}^*\mathcal O_{_{X_0}}(x_0)\rightarrow \mathcal P_{_{x_0}}\rightarrow 0
\end{equation}

Using the short exact sequence, we get the following isomorphism of the determinant of cohomologies.
 
\begin{equation}
\text{Det}~(\mathcal P\otimes p_{_{X_0}}^*\mathcal O_{_{X_0}}(x_0))\cong \text{Det}~\mathcal P\otimes \mathcal P_{_{x_0}}^{\otimes -1}
\end{equation}

By repeating this, we get the following
\begin{equation}
\text{Det}~(\mathcal P\otimes p_{_{X_0}}^*\mathcal O_{_{X_0}}((g-k-1)x_0))\cong \text{Det}~~\mathcal P\otimes \mathcal P_{_{x_0}}^{\otimes -(g-k-1)}
\end{equation}
Now notice that $\mathcal P\otimes p_{_{X_0}}^*\mathcal O_{_{X_0}}((g-k-1)x_0)$ is a Poincar\'e family of line bundles of degree $(g-k-1)$ on $X_0$ parametrized by $J_0$. Therefore from \cite[Lemma 2.4]{8}, it follows that
\begin{equation}
\text{Det}~\mathcal P\otimes \mathcal P_{_{x_0}}^{\otimes -(g-k-1)}\cong \text{Det}~(\mathcal P\otimes p_{_{X_0}}^*\mathcal O_{_{X_0}}((g-k-1)x_0))\cong \mathcal O_{_{J_0}}(\Theta_0)
\end{equation}
This proves the first statement.

From \cite[Corollary 14]{14}, it follows that $\text{Det}~~(\mathcal F\otimes p_k^*\mathcal O_{_{X_k}}((g-1)x_0))\cong \mathcal O_{_{\overline{J}_{X_k}}}(\Theta_k)$. Now the statement $(2)$ follows from similar arguments as above.

To prove $(3)$, consider the following exact sequence of sheaves over $X_k\times \widetilde{J_k}$
\begin{equation}\label{Ex2}
0\rightarrow \mathcal F'\rightarrow (q_k)_*\mathcal P\rightarrow \oplus_{i=1}^k Q_i\rightarrow 0
\end{equation}

The torsion-free sheaf $\mathcal F'$ is a family of torsion-free sheaves of rank $1$ and degree $0$ over $X_k$ parametrized by $\widetilde{J_k}$. Therefore, from \eqref{Induce1} it follows that there exists a Poincar\'e family $\mathcal F''$ of torsion free sheaves of rank $1$ and degree $0$ over $X_k$ parametrized by $\overline{J}_k$ such that $\nu_k^*\mathcal F''\cong \mathcal F'$, where $\nu_k: X_k\times \widetilde{J_k}\rightarrow X_k\times \overline{J}_k$ is the natural map. From $(2)$ we have that

\

\begin{equation}
\text{Det}~\mathcal F \otimes \mathcal F^{\otimes -(g-1)}_{x_0}\cong \text{Det}~\mathcal F'' \otimes \mathcal F''^{\otimes -(g-1)}_{x_0}
\end{equation}

\

Then by the functoriality of the determinant of cohomology, we get

\begin{equation}\label{Eq10}
\nu_k^*(\text{Det}~\mathcal F \otimes \mathcal F^{\otimes -(g-1)}_{x_0})\cong \nu_k^*(\text{Det}~\mathcal F'' \otimes \mathcal F''^{\otimes -(g-1)}_{x_0})\cong \text{Det}~\mathcal F'\otimes \mathcal F'^{\otimes -(g-1)}_{x_0})
\end{equation}

\

From \eqref{Ex2}, we get
\begin{equation}
\text{Det}~\mathcal F'\cong \text{Det}~\mathcal P\otimes (\otimes_{_{i=1}}^k Q_i)
\end{equation}
Now combining this with equation \eqref{Eq10} we get,

\

\begin{equation}
\nu_k^*(\text{Det}~\mathcal F \otimes \mathcal F^{\otimes -(g-1)}_{x_0})\cong \text{Det}~\mathcal P\otimes (\otimes_{_{i=1}}^k Q_i)\otimes \mathcal P^{\otimes -(g-1)}_{x_0})
\end{equation}

\

Therefore, from the previous lemma, it follows that 

$$\nu_k^*(\text{Det}~\mathcal F\otimes \mathcal F_{_{x_0}}^{\otimes (g-1)})\cong \nu_k^*(\text{Det}~\mathcal F'\otimes \mathcal F'^{\otimes (g-1)}_{_{x_0}})\cong \text{Det}~\mathcal P\otimes \mathcal P^{\otimes (g-1)}_{_{x_0}}\otimes (\otimes_{i=1}^k Q_i)$$.

The fact that the line bundles in $(1)$ and $(2)$ are ample follows from \cite[Sect. 17, p. 163]{10} and \cite[Theorem 7]{5}. Since the morphism $\nu_k: \widetilde{J_k}\rightarrow \overline J_k$ is a finite morphism and

$$\nu_k^*(\text{Det}~\mathcal F \otimes \mathcal F^{\otimes -(g-1)}_{x_0})\cong \text{Det}~\mathcal P\otimes \mathcal P^{\otimes -(g-1)}_{x_0}\otimes (\otimes_{i=1}^k Q_i),$$

therefore it is ample over $\widetilde{J_k}$.

\end{proof}

\section{\textbf{A specialization of the compactified Jacobian of a nodal curve with a single node}}

\

Let $X_0$ be a smooth projective curve. Let us choose and fix a point $z\in X_0$. By a \textbf{general point} $x\in X_0$ we mean that $x\neq z$. In this section, we will construct an algebraic family $\mathcal{J}_1$ over $X_0$ such that the fiber over a general point $x\in X_0$ is isomorphic to $\overline{J}_{_{X(x)}}$ (see Notations) and the fiber over $z$ is isomorphic to $J_{_{X_0}}\times R$. By definition, it is, therefore, a specialization of $\overline{J}_{_{X(x)}}$ to $J_{_{X_0}}\times R$.

\

\subsection{\textbf{The construction of the total space}}

\

We will construct the family $\mathcal{J}_1$ as a push-out of the following $\mathbb P^1$-bundle.
\begin{equation}\label{define12}
\widetilde{\mathcal{J}}_1:=\mathbb P(\mathcal P\oplus p_2^*\mathcal P_{z}), ~~\text{over}~~  X_0\times J_0,
\end{equation}
where 
\begin{enumerate}
\item $\mathcal P$ is a Poincar\'e line bundle over $X_0\times J_0$,
\item $p_2:X_0\times J_0\rightarrow J_0$ is the projection morphism. 
\item $\mathcal P_z$ denotes the line bundle over $J_0$ obtained by restricting $\mathcal P$ to the closed sub-scheme $z\times J_0$ and by identifying $z\times J_0$ with $J_0$.
\end{enumerate}

\

\begin{rema}\label{uniquo}
The variety $\widetilde{\mathcal J_1}$ parametrises tuples $(x, L, L_x\oplus L_z\rightarrow Q)$, where
\begin{enumerate}
\item $x$ is a point of $X_0$,
\item $L$ is a line bundle of degree $0$ over $X_0$, 
\item $L_x\oplus L_z\rightarrow Q$ is a $1$-dimensional quotient.
\end{enumerate}
\end{rema}

\

\subsection{\textbf{Two natural divisors on the total space}} \label{twodiv}

\

The $\mathbb P^1$ bundle has two natural sections $\mathcal D_1$ and $\mathcal D'_1$, which correspond to the two following natural quotients $\mathcal P\oplus p_2^*\mathcal P_z\rightarrow \mathcal P$ and $\mathcal P\oplus p_2^*\mathcal P_z\rightarrow p_2^*\mathcal P_z$ respectively. Being sections, these two divisors are both isomorphic to $X_0\times J_0$ (the isomorphism is given by the restrictions of the projection morphism to $\widetilde{\mathcal{J}}_1\to X_0\times J_0$).

\

\begin{rema}\label{uni1}
The variety $\mathcal D_1$ parametrises tuples $(x, L, L_x\oplus L_z\rightarrow L_x)$, and the variety $\mathcal D'_1$ parametrises tuples $(x, L, L_x\oplus L_z\rightarrow L_z)$. Notice when $x=z$, there is an ambiguity about the quotients $L_x\oplus L_z\rightarrow L_x$ and $L_x\oplus L_z\rightarrow L_z$. To resolve this, we refer to $L_x\oplus L_z\rightarrow L_x$ as the first quotient and to  $L_x\oplus L_z\rightarrow L_z$ as the second quotient. We see that the varieties $\widetilde{\mathcal J_1}, \mathcal D_1$, and $\mathcal D'_1$ have universal properties because they parametrize the tuples, described above. 
\end{rema}

\

\begin{lema}\label{Dis1}
$\mathcal D_1\cap \mathcal D'_1=\emptyset$.
\end{lema}
\begin{proof}
For any point $t\in X_0$, the fibers at $t$ of the two natural sections of $\mathbb P(\mathcal P \oplus p_{_2}^*\mathcal P_{z}) $ are  $$\mathcal P_t \oplus \mathcal P_{z}\rightarrow \mathcal P_t$$ and $$\mathcal P_t \oplus \mathcal P_{z}\rightarrow \mathcal P_{z}$$ They are obviously distinct quotients. Therefore $\mathcal D_1\cap\mathcal D'_1=\emptyset.$
\end{proof}

\

\subsection{\textbf{Twisted isomorphism between the divisors $\mathcal D_1$ and $\mathcal D'_1$}} 

\

Notice that $\mathcal D_1$ and $\mathcal D'_1$ are abstractly isomorphic to $X_0\times J_0$ because they are sections of the morphism $\widetilde{\mathcal{J}}_1\rightarrow X_0\times J_0$. Proposition \ref{Involution1} shows that there is another natural isomorphism (``twisted isomorphism'') between these two divisors \eqref{twist1012}. To state the proposition, we need to fix some notation and a preparatory lemma.

\subsubsection{\textbf{Notation.}}
\begin{equation}\label{NotDiag}
\begin{tikzcd}
& X_0\times \widetilde{\mathcal J_1} \arrow{d}{Id\times \widetilde{\pi_1}}\\
& X_0\times X_0\times J_0\arrow{dl}{q}\arrow{dd}{r}\arrow{dr}{s_i}\\
X_0\times J_0 && X_0 \\
&X_0\times X_0
\end{tikzcd}
\end{equation}

where
\begin{enumerate}
 \item $\widetilde{\pi_1}$ is the projection $\widetilde{\mathcal J_1}\rightarrow X_0\times J_0$, and $Id\times \widetilde{\pi_1}: X_0\times \widetilde{\mathcal J_1}\rightarrow X_0\times (X_0\times J_0)$ is the product of the identity morphism on the first factor and $\widetilde{\pi_1}$ on $\widetilde{\mathcal J_1}$,
 \item $q: X_0\times X_0\times J_0\rightarrow X_0\times J_0$ is the projection $(x_1,x_2, L)\mapsto (x_1, L)$,
 \item $r: X_0\times X_0\times J_0\rightarrow X_0\times X_0$ is the projection $(x_1,x_2, L)\mapsto (x_1, x_2)$,
 \item $s_1: X_0\times X_0\times J_0\rightarrow X_0$ and $s_2: X_0\times X_0\times J_0\rightarrow X_0$ are the projections onto the first and second $X_0$, respectively, 
\item $\widetilde{q}:=q\circ (Id\times \widetilde{\pi_1})$, $\widetilde{r}:=r\circ (Id\times \widetilde{\pi_1})$ and $\widetilde{s_i}:=s_i\circ (Id\times \widetilde{\pi_1})$.
 \end{enumerate}

\

\begin{lema}\label{iden121}
\begin{enumerate}
\item $\widetilde{r}^{-1}(\Delta)\cong \widetilde{\mathcal J_1}$,
\item Let us denote by $j: \mathcal D_1\hookrightarrow \widetilde{\mathcal J_1}$ and $j': \mathcal D'_1\hookrightarrow \widetilde{\mathcal J_1}$ the natural inclusion maps. Then  $\widetilde{r}^{-1}(\Delta)\cap (X_0\times \mathcal D_1)\cong (Id\times j)^{-1}\circ \widetilde{r}^{-1}(\Delta)  \cong \mathcal D_1$, and $\widetilde{r}^{-1}(\Delta)\cap (X_0\times \mathcal D'_1)\cong (Id\times j')^{-1}\circ \widetilde{r}^{-1}(\Delta)  \cong \mathcal D'_1$.
\end{enumerate}
\end{lema}

\

\begin{proof}
We have the following diagram
\begin{equation}
\begin{tikzcd}
\widetilde{\mathcal J_1} \arrow{d} & X_0\times \widetilde{\mathcal J_1}\arrow{d}\arrow{l} & \widetilde{r}^{-1}(\Delta)\arrow{l}\arrow{d}\\
X_0\times J_0 & X_0\times X_0\times J_0\arrow{l} &  r^{-1}(\Delta)\arrow{l}
\end{tikzcd}
\end{equation}

Notice that the two squares are Cartesian. Therefore, the composite of the two squares is also Cartesian. Now the first statement follows from the observation that the composite map $r^{-1}(\Delta)\rightarrow X_0\times J_0$ is an isomorphism. 

Since $r^{-1}(\Delta)\rightarrow X_0\times J_0$, therefore $\widetilde{r}^{-1}(\Delta)\cap (X_0\times \mathcal D_1)\cong (Id\times j)^{-1}\circ \widetilde{r}^{-1}(\Delta)  \cong \mathcal D_1$. The other statement follows similarly.
\end{proof}

\

\begin{prop}\label{Involution1}
There is a natural isomorphism $\tau:\mathcal D_1\rightarrow \mathcal {D}'_1$ given by 
\begin{equation}\label{sett1}
(x,L, L_{x}\oplus L_{z}\rightarrow L_{x})\mapsto (x,L':=L\otimes\OO_X(z-x), L'_{x}\oplus L'_{z}\rightarrow L'_{z}).
\end{equation}
 Here for a line bundle $M$ over $X_0$, we denote by $M_{x}\oplus M_{z}\rightarrow M_{x}$ the first projection and by $M_{x}\oplus M_{z}\rightarrow M_{z}$ the second projection.
\end{prop}

\

\begin{proof}
Consider the line bundle over $X_0\times \widetilde{\mathcal J_1}$
\begin{equation}
\mathcal{P'}:=\widetilde{q}^*\mathcal P\otimes \widetilde{r}^*\mathcal O_{X_0\times X_0}(-\Delta)\otimes \widetilde{s_1}^*\mathcal O_{X_0}(z),
\end{equation}

where $\Delta$ is the subvariety $\{(x,x)|x\in X_0\}\subset
X_0\times X_0$. It is straightforward to check that $\mathcal {P'}$ is a Poincar\'e line bundle $\mathcal{P'}$ i.e., a family
of degree $0$ line bundles over $X_0$ parametrized by $\widetilde{\mathcal J_1}$. 

Over $\widetilde{\mathcal J_1}$ we have a universal quotient
\begin{equation}\label{quot12}
(\widetilde{q}^* \mathcal P)|_{\widetilde{r}^{-1}(\Delta)}\oplus (\widetilde{q}^* \mathcal P)|_{z\times \widetilde{\mathcal J_1}}\rightarrow \mathcal L
\end{equation}

Here $\mathcal L$ denotes the universal quotient line bundle whose fibers are as described in Remark \ref{uniquo}. The fiber of this line bundle at a point $[(x, L, L_x\oplus L_z\rightarrow Q)]\in \widetilde{\mathcal J_1}$ is $Q$.

Notice that by lemma \ref{iden121}, $\widetilde{r}^{-1}(\Delta)$ and $z\times \widetilde{\mathcal J_1}$ both can be identified with $\widetilde{\mathcal J_1}$.

Over $\mathcal D_1$, the quotient \eqref{quot12} becomes 
\begin{equation}
(\widetilde{q}^* \mathcal P)|_{\widetilde{r}^{-1}(\Delta)\cap (X_0\times \mathcal D_1)}\oplus (\widetilde{q}^* \mathcal P)|_{z\times \mathcal D_1}\rightarrow (\widetilde{q}^* \mathcal P)|_{\widetilde{r}^{-1}(\Delta)\cap (X_0\times \mathcal D_1)}
\end{equation}

Notice that by lemma \ref{iden121}, $\widetilde{r}^{-1}(\Delta)\cap (X_0\times \mathcal D_1)$ and $z\times \mathcal D_1$ both can be identified with $\mathcal D_1$.

We have another natural quotient line bundle over $\mathcal D_1$ which is the following.

\begin{equation}\label{modiquo}
\mathcal{P'}|_{\widetilde{r}^{-1}(\Delta)\cap (X_0\times \mathcal D_1)}\oplus\mathcal{P'}|_{z\times \mathcal D_1}\rightarrow \mathcal{P'}|_{z\times \mathcal D_1}.
\end{equation}
 
By the respective universal properties (remark \ref{uni1}) of $\mathcal D_1 $ and $\mathcal D'_1$, the above modified quotient \eqref{modiquo} induces an isomorphism $\tau:\mathcal D_1\rightarrow \mathcal D'_1$. It is straightforward to check that this isomorphism has the desired property \eqref{sett1}.
\end{proof}

\

\subsection{\textbf{The construction of the specialization by push-out and its singularities}}
Consider the following topological quotient space
\begin{equation}\label{Quo1}
\mathcal J_1:=\frac{\widetilde{\mathcal J_1}}{\mathcal D_1\sim_{\tau} \mathcal D'_1},
\end{equation}
where $\mathcal D_1\sim_{\tau} \mathcal D'_1$ means that $\mathcal D_1$ and $\mathcal D'_1$ are identified via the automorphism $\tau$ \eqref{sett1}. We denote by $\nu_1: \widetilde{\mathcal J_1}\rightarrow \mathcal J_1$ the quotient map.  We denote by $\mathcal V$ the image of $\mathcal D_1$, or equivalently the image of $\mathcal D'$ under the quotient map $\nu_1$.

\

\begin{thm}\label{Sing1}
$\mathcal J_1$ is an algebraic space with only normal crossing singularities.
\end{thm}

\begin{proof} 
From proposition \cite[Theorem 3.1]{1} and \cite[Theorem 45 (Gluing of algebraic spaces).]{K}, it follows that $\mathcal{J}_1$ is an algebraic space. To determine the singularities of $\mathcal{J}_1$, we consider the following exact sequence of sheaves
\begin{equation}\label{ke12}
0\rightarrow \mathcal O^{}_{_{\mathcal J_1}}\rightarrow (\nu_1)_*\mathcal O^{}_{_{\widetilde{\mathcal J_1}}}
\rightarrow \mathcal O^{}_{_{\mathcal V}}\rightarrow 0
\end{equation}
Let $v$ be a point of $\mathcal V$, and let $v_1\in \mathcal D_1$ and $v_2\in \mathcal D'_1$ denote the pre-images of $v$ under the map $\nu_1$. Then \eqref{ke12} induces the following short exact sequence of analytic local rings

$$
\begin{tikzcd}
0\arrow{r} & \hat{\mathcal O}_{_{\mathcal J_1, v}}\arrow{r}\arrow{d}{=} & \hat{\mathcal O}_{_{\widetilde{\mathcal J_1}, v_1}}\oplus \hat{\mathcal O}_{_{\widetilde{\mathcal J_1}, v_2}}\arrow{r}\arrow{d}{\cong} & \hat{\mathcal O}_{_{\mathcal V, v}}\arrow{r}\arrow{d}{\cong}& 0\\
0\arrow{r} & \hat{\mathcal O}_{_{\mathcal J_1, v}}\arrow{r}& k[|x_1,\dots,x_{n-1}, x_n|]\oplus k[|x_1,\dots,x_{n-1}, x_{n+1}|] \arrow{r} & k[|x_1,\dots, x_{n-1}|] \arrow{r} & 0
\end{tikzcd}
$$

The morphism $$k[|x_1,\dots,x_{n-1}, x_n|]\oplus k[|x_1,\dots,x_{n-1}, x_{n+1}|] \rightarrow k[|x_1,\dots, x_{n-1}|]$$ is given by $$(f,g)\mapsto f (\text{mod}~ x_n)-g (\text{mod}~ x_{n+1}).$$
Hence we see that $$\hat{\mathcal O}_{_{\mathcal J_1,v}}\cong \frac{k[|x_1, \dots, x_{n-1}, x_n, x_{n+1}|]}{x_n\cdot x_{n+1}}.$$ Therefore, the algebraic space $\mathcal J_1$ has normal crossing singularities along $\mathcal V$.

\end{proof}

\

We have the following commutative diagram:
\begin{equation}
\begin{tikzcd}
\mathcal D_1\cup \mathcal D'_1\arrow{d} \arrow{r} &  \widetilde{\mathcal J}_1\arrow{d}\arrow[bend left]{ddr}\\
\mathcal V\arrow{r}\arrow[bend right]{drr} & \mathcal J_1\arrow[dotted]{dr}\\
&& X_0
\end{tikzcd}
\end{equation}
Therefore we have a projection morphism $\mathcal J_1\rightarrow X_0$ from the push-out to $X_0$. Let us denote it by $\mathfrak f_1: \mathcal J_1\rightarrow X_0$.

\

\subsection{\textbf{Construction of the theta bundle on $\mathcal J_1$}}

Consider the following diagram

\begin{equation}
\begin{tikzcd}
\widetilde{\mathcal J_1}\arrow{d}{\widetilde{\pi_1}}\\
X_0\times J_0\arrow{r}{p_1}\arrow{d}{p_2}& X_0\\
J_0
\end{tikzcd}
\end{equation}

\

Let us define $\widetilde{p_1}:=p_1\circ \widetilde{\pi_1}$ and $\widetilde{p_2}:=p_2\circ \widetilde{\pi_1}$. Let us now choose a point $p_0$ in $X_0$ such that $p_0\neq z$. We can define an ample line bundle $\Theta_0$ on $J_0$ as in proposition \ref{Ample1}. We will show that the line bundle $\widetilde{p_2}^*\Theta_0\otimes \mathcal L$ is a relatively ample line bundle for the morphism $\widetilde{p_1}: \widetilde{\mathcal J_1}\rightarrow X_0$. Here $\mathcal L$ denotes the universal quotient line bundle whose fibers are as described in Remark \ref{uniquo}. The fiber of this line bundle at a point $[(x, L, L_x\oplus L_z\rightarrow Q)]\in \widetilde{\mathcal J_1}$ is $Q$. From proposition \ref{Ample1}, we have $\widetilde{p_2}^*\Theta_0\otimes \mathcal L=\widetilde{p_2}^*\text{Det}~\mathcal P\otimes \widetilde{p_2}^*\mathcal P_{p_0}\otimes \mathcal L$.

\

Consider the commutative square

\begin{equation}
\begin{tikzcd}
X_0\times \widetilde{\mathcal J_1}\arrow{d}{\widetilde{q}}\arrow{r} & \widetilde{\mathcal J_1}\arrow{d}{\widetilde{p_2}}\\
X_0\times J_0\arrow{r}{p_2} & J_0
\end{tikzcd}
\end{equation} 

From the above diagram, we easily see the following
\begin{enumerate}
\item $\text{Det}~~\widetilde{q}^*\mathcal P\cong \widetilde{p_2}^*\text{Det}~~\mathcal P$,
\item $(\widetilde{q}^*\mathcal P)|_{p_0\times \widetilde{\mathcal J_1}}\cong \widetilde{p_2}^*\mathcal P_{p_0}$.
\end{enumerate}

\

The above isomorphisms lead to the following definition.

\

\begin{defe}\label{Thet}
We define a line bundle over $\widetilde{\mathcal J_1}$
\begin{equation}
\widetilde{\Theta_1}:= \text{Det}~\big(\widetilde{q}^*\mathcal P \otimes \mathcal L\otimes (\widetilde{q}^*\mathcal P)|_{p_0\times \widetilde{\mathcal J_1}}^{\otimes {-(g-1)}}\big ),
\end{equation}
where $\text{Det}(-)$ denotes the determinant of cohomology. We refer to this line bundle as the \textbf{theta bundle over $\widetilde{\mathcal J_1}$.}
\end{defe}

\

Notice that by (1) and (2) above, we have 

\

\begin{equation}
\widetilde{\Theta_1}\cong \widetilde{p_2}^*\text{Det}~~\mathcal P\otimes \mathcal L\otimes \widetilde{p_2}^*\mathcal P_{p_0}
\end{equation}

\

\subsection{\textbf{Some properties of the isomorphism $\tau: \mathcal D_1\rightarrow \mathcal D'_1$}}

\
We want to show that the line bundle $\widetilde{\Theta_1}$, defined above, is relatively ample for the morphism $\widetilde{\mathcal J_1}\rightarrow X_0$. Moreover, there is an open subset $B^0_1$ (yet to be defined) of $X_0$ such that the line bundle $\widetilde{\Theta_1}$ descends to the base change of $\widetilde{\mathcal J_1}|_{B^0_1}$. But before that, we list out a few results in this subsection which will be useful to determine the pull-backs by the isomorphism $\tau$ of several natural line bundles on $\widetilde{\mathcal J_1}$.

\

\begin{lema} Let $p_{23}: X_0\times X_0\times J_0\rightarrow X_0\times J_0$ denote the projection onto the product of the second and the third factors. Let us consider the cartesian square

\begin{equation}\label{square123}
\begin{tikzcd}
(x_1, x_2, L)\arrow{rrrr}&&&& (x_1, x_2, L\otimes \mathcal O_{X_0}(-x_2+z))\\
X_0\times X_0\times J_0\arrow{rrrr}{Id\times \tau}\arrow{dd}{p_{23}} &&&& X_0\times X_0\times J_0\arrow{dd}{p_{23}}\\ 
&&&&\\
X_0\times J_0\arrow{rrrr}{\tau} &&&& X_0\times J_0\\
(x_2, L)\arrow{rrrr}&&&& (x_2, L\otimes \mathcal O_{X_0}(-x_2+z))
\end{tikzcd}
\end{equation}

Let $\mathcal P$ be a Poincar\'e line bundle on $X_0\times J_0$. Then, we have the following.
\begin{enumerate}
\item $(Id\times \tau)^*(q^*\mathcal P)\cong q^*\mathcal P\otimes r^*\mathcal O_{X_0\times X_0}(-\Delta)\otimes s_1^*\mathcal O_{X_0}(z)$ ~~\text{over}~~$X_0\times X_0\times J_0$,
\item for any point $p\in X_0$, $\tau^*(p_2^*\mathcal P_{p})\cong p_2^*\mathcal P_{p}\otimes p_1^*\mathcal O_{X_0}(-p)$~~\text{over}~~$X_0\times J_0$,

\item $\tau^* \text{Det}~~q^*\mathcal P\cong\text{Det}~~(Id\times \tau)^*(q^*\mathcal P)$~~\text{over}~~$X_0\times J_0$,
\item $\text{Det}~~(Id\times \tau)^*(q^*\mathcal P)\cong \text{Det}~~(q^*\mathcal P)\otimes (q^*\mathcal P)|_{r^{-1}(\Delta)}\otimes p_{2}^*\mathcal P_{z}^{-1}\otimes p_{1}^*\mathcal O_{X_0}(z)$, ~~\text{over}~~$X_0\times J_0$. 
\end{enumerate}
\end{lema}

\

\begin{proof}
The statement (1) follows from the universal property of $J_0$ and the definition of the map $\tau$, as in the diagram \eqref{square123}. 

To prove (2), consider the following diagram

\begin{equation}\label{twist15}
\begin{tikzcd}
X_0\times J_0\arrow{r}{\tau} & X_0\times J_0\arrow{d}{p_2}\\
& J_0
\end{tikzcd}
\end{equation}

Since we want to compute $\tau^*\circ p_2^* \mathcal P_{p}$, we consider the following diagram instead of \eqref{twist15}. 

\begin{equation}
\begin{tikzcd}
X_0\times X_0\times J_0\arrow{r}{Id \times \tau} & X_0\times X_0\times J_0\arrow{d}{q}\\
& X_0\times J_0
\end{tikzcd}
\end{equation}

Then $\tau^*\circ p_2^* \mathcal P_{p}$ is isomorphic to $((Id \times \tau)^*\circ q^*\mathcal P)|_{p\times X_0\times J_0}$. From (1), we have 

$$(Id\times \tau)^*(q^*\mathcal P)\cong q^*\mathcal P\otimes r^*\mathcal O_{X_0\times X_0}(-\Delta)\otimes s_1^*\mathcal O_{X_0}(z).$$ 

Therefore,

$$
\tau^*\circ p_2^* \mathcal P_{p}\cong (q^*\mathcal P)|_{p\times X_0\times J_0}\otimes (r^*\mathcal O_{X_0\times X_0}(-\Delta))|_{p\times X_0\times J_0}\otimes (s_1^*\mathcal O_{X_0}(z))|_{p\times X_0\times J_0}
$$
\begin{equation}
\cong p_2^*\mathcal P_{p}\otimes p_1^*\mathcal O_{X_0}(-p)
\end{equation}
This proves (2).

The statement (3) also follows from the commutative square \ref{square123}.

To see (4), define $\mathcal P'':=q^*\mathcal P\otimes r^*\mathcal O_{X_0\times X_0}(-\Delta)$. Hence, $\mathcal P'=\mathcal P''\otimes s_1^*\mathcal O_{X_0}(z)$.

Consider the following short exact sequence of sheaves over $X_0\times X_0\times J_0$
\begin{equation}
0\rightarrow \mathcal P''\rightarrow \mathcal P''\otimes s_1^*\mathcal O_{X_0}(z)\rightarrow \mathcal P''|_{z\times X_0\times J_0}\rightarrow 0
\end{equation}

Therefore, we have,

\begin{equation}
\text{Det}~\mathcal P'\cong \text{Det}~\mathcal P''\otimes (\mathcal P''|_{z\times X_0\times J_0})^{-1}
\end{equation}

Notice 
\begin{equation}
\mathcal P''|_{z\times X_0\times J_0}\cong (q^*\mathcal P)|_{z\times X_0\times J_0}\otimes (r^*\mathcal O_{X_0\times X_0}(-\Delta))|_{z\times X_0\times J_0}\cong p_2^*\mathcal P_z\otimes p_1^*\mathcal O_{X_0}(-z).
\end{equation}

Therefore, 

\begin{equation}
\text{Det}~\mathcal P'\cong \text{Det}~\mathcal P''\otimes (p_2^*\mathcal P_z)^{-1}\otimes p_1^*\mathcal O_{X_0}(z)
\end{equation}

Now let us compute $\text{Det}~\mathcal P''$. Consider the following short exact sequence 
\begin{equation}
0\rightarrow \mathcal P''\rightarrow q^*\mathcal P\rightarrow (q^*\mathcal P)|_{r^{-1}(\Delta)}\rightarrow 0
\end{equation}

Therefore, we get 
\begin{equation}
\text{Det}~\mathcal P''\cong \text{Det}~q^*\mathcal P\otimes (q^*\mathcal P)|_{r^{-1}(\Delta)}
\end{equation}
and 
\begin{equation}
\text{Det}~\mathcal P'\cong \text{Det}~q^*\mathcal P\otimes (q^*\mathcal P)|_{r^{-1}(\Delta)} \otimes (p_2^*\mathcal P_z)^{-1}\otimes p_1^*\mathcal O_{X_0}(z)
\end{equation}
This completes the proof.
\end{proof}

\

\begin{rema}\label{rules}
The above statement also holds if we replace the above square \ref{square123} with the following square 
\begin{equation}
\begin{tikzcd}
X_0\times \mathcal D_1\arrow{rrrr}{Id\times \tau}\arrow{dd} &&&& X_0\times \mathcal D'_1\arrow{dd}\\ 
&&&&\\
\mathcal D_1\arrow{rrrr}{\tau} &&&& \mathcal D'_1
\end{tikzcd}
\end{equation}
where $\tau$ is the isomorphism defined in lemma \ref{Involution1}. We list out some more statements here which will be useful in further discussions.
\begin{enumerate}
\item $(Id\times \tau)^*(\widetilde{q}^*\mathcal P)\cong \widetilde{q}^*\mathcal P\otimes \widetilde{r}^*\mathcal O_{X_0\times X_0}(-\Delta)\otimes \widetilde{s_1}^*\mathcal O_{X_0}(z)$ over $X_0\times \mathcal D_1$,
\item for any point $p\in X_0$, we have $\tau^*(\widetilde{p_{2}}^*\mathcal P_{p})\cong \widetilde{p_{2}}^*\mathcal P_{p}\otimes \widetilde{p_{1}}^*\mathcal O_{X_0}(-p)$ over $\mathcal D_1$,

\item $\tau^* (\text{Det}~~\widetilde{q}^*\mathcal P)\cong\text{Det}~~(Id\times \tau)^*(\widetilde{q}^*\mathcal P)$ over $\mathcal D_1$,

\item $\text{Det}~~(Id\times \tau)^*(\widetilde{q}^*\mathcal P)\cong \text{Det}~~(\widetilde{q}^*\mathcal P)\otimes (\widetilde{q}^*\mathcal P)|_{\widetilde{r}^{-1}(\Delta)}\otimes \widetilde{p_2}^*\mathcal P_{z}^{-1}\otimes \widetilde{p_1}^*\mathcal O_{X_0}(z)$ over $\mathcal D_1$.
\end{enumerate}
It is necessary to clarify the notation tilde $~~\widetilde{}~~$. We remind here that if $p$ is a projection map from $X_0\times J_0\to J_0$, we denote the composition of $\widetilde{\mathcal J_1}\rightarrow X_0\times J_0$ with the projection $p$ by $\widetilde{p}$. 
\end{rema}

\

\subsection{\textbf{Relative ampleness of $\widetilde{\Theta_1}$}}

In this subsection, we will show that the line bundle $\widetilde{\Theta_1}$, defined in Definition \ref{Thet}, is relatively ample for the morphism $\widetilde{\mathfrak f_1}:\widetilde{\mathcal J_1}\rightarrow X_0$. 

\

\begin{prop}
The line bundle $\widetilde{\Theta_1}$ is ample relative to the morphism $\widetilde{\mathfrak f_1}:\widetilde{\mathcal J_1}\rightarrow X_0$.
\end{prop}

\

\begin{proof}
Since $\widetilde{\mathfrak f_1}$ is projective, it is enough to show that the restriction of $\widetilde{\Theta_1}$ to the fiber over every point $x\in X_0$ is ample.
For any $x\in X_0$, $\widetilde{\mathfrak f_1}^{-1}(x)=\mathbb P(\mathcal P_x\oplus \mathcal P_z)$ \eqref{define12}. Let us denote it by $\mathbb P(x)$. Then the restriction of $\widetilde{\Theta_1}$ to $\mathbb P(x)$ is isomorphic to

\begin{align*}
\text{Det}~(\widetilde{q}^*\mathcal P)|_{\mathbb P(x)} \otimes \mathcal L|_{\mathbb P(x)}\otimes (\widetilde{q}^*\mathcal P)|_{p_0 \times \mathbb P(x)}^{\otimes {-(g-1)}}\\
\cong (\text{Det}~(\widetilde{q}^*\mathcal P)|_{\mathbb P(x)} \otimes (\widetilde{q}^*\mathcal P)|_{{p_0} \times {\mathbb P(x)}}^{\otimes {-(g-2)}})\otimes \mathcal L|_{\mathbb P(x)}\otimes \widetilde{q}^*\mathcal P|_{{p_0} \times {\mathbb P(x)}}^{\otimes {-1}}\\
\cong \tilde p_{2,x}^*\Theta_0(x)\otimes \mathcal O_{\mathbb P(x)}(1)\otimes \widetilde{q}^*\mathcal P|_{{p_0} \times {\mathbb P(x)}}^{\otimes {-1}},
\end{align*}

\begin{equation}\label{theta123}
\end{equation}

where 
\begin{enumerate}
\item $\tilde p_{2,x}: \mathbb P(x)\rightarrow J_0$ denotes the projection map,
\item $\Theta_0(x)$ denotes the theta bundle (proposition \ref{Ample1}) over $\mathbb P(x)$,
\item $\mathcal O_{\mathbb P(x)}(1)$ is the tautological bundle on the projective bundle of $\mathbb P(\mathcal P_x\oplus \mathcal P_z)$.
\end{enumerate}

It can be easily seen that $\widetilde{q}^*\mathcal P|_{X_0\times \mathbb P(x)}$ is isomorphic to the pull-back of the Poincar\'e bundle $\mathcal P$ by the map $Id\times \widetilde{p_{2,x}}:X_0\times \mathbb P(x)\rightarrow X_0\times J_0$ which is identity on the first factor and $\widetilde{p_{2,x}}$ on $\mathbb P(x)$. Notice that

\begin{align*}
E(x):=(\widetilde{p_{2,x}})_*({\widetilde{p_{2,x}}}^*\Theta_0(x)\otimes \mathcal O_{\mathbb P(x)}(1))\\
\cong \Theta_0(x)\otimes ({\widetilde{p_{2,x}}})_*\mathcal O_{\mathbb P(x)}(1)\\
\cong \Theta_0(x)\otimes (\mathcal P_x\oplus \mathcal P_z)\\\cong (\Theta_0\otimes \mathcal P_x)\oplus (\Theta_0\otimes \mathcal P_z).
\end{align*}

\

Since each of the direct summands is an ample line bundle, the vector bundle $E(x)$ is ample. Then by \cite[Theorem 3.2]{6}, it follows that $\OO_{\mathbb{P}(E(x))}(1)$ is ample line bundle over $\mathbb P(E(x))$, which is isomorphic to $\mathbb P(\Theta_0(x)\otimes \mathcal P_x\oplus \Theta_0(x)\otimes \mathcal P_z)\cong \mathbb P(\mathcal P_x\oplus \mathcal P_z)$. Therefore the line bundle ${\widetilde{p_{2,x}}}^*\Theta_0(x)\otimes \mathcal O_{\mathbb P(x)}(1)$ is ample over $\mathbb P(x)$ for any $x\in X_0$. Now notice that the line bundle $(\widetilde{q}^*\mathcal P)|_{{p_0} \times {\mathbb P(x)}}$ in \eqref{theta123} is isomorphic to the pullback of the line bundle $ \mathcal P_{p_0}$ by the map $\mathbb P(x)\rightarrow J_{X_0}$ and hence it is algebraically equivalent to the trivial line bundle. Therefore $\widetilde{\Theta_1}|_{\mathbb P(x)}$ is ample for any $x\in X_0$.
\end{proof}

\

\subsection{\textbf{Descent of the line bundle $\widetilde{\Theta_1}$}}
In the previous subsection, we have shown that the line bundle $\widetilde{\Theta_1}$ is ample relative to $\widetilde{\mathcal J_1}\to X_0$. In this subsection, we investigate whether the line bundle $\widetilde{\Theta_1}$ descends to the quotient $\widetilde{\mathcal J_1}\to \mathcal J_1$. But it turns out that the line bundle only descends when we restrict the family $\mathcal J_1\to X_0$ to the following smaller base instead of $X_0$.  

\

Consider the open subset $B^o_1:=X_0\setminus \{p_0\}$. We take the base change of $\widetilde{\mathcal J_1}$ over the open set $B^o_1$. By abuse of notation, we also denote it by $\widetilde{\mathcal J_1}$. From here onwards we will similarly base change everything on $B^o_1$ and denote them by the same notations.

\

\begin{thm} \label{Descent} \label{D2024}
\begin{enumerate}
\item The morphism $\mathfrak f_1:\mathcal J_1\rightarrow B^o_1$ is projective.
\item The fibers of the morphism have the following description.
\begin{equation}
\mathfrak f_1^{-1}(x)\cong \left\{
\begin{array}{@{}ll@{}}
&\overline{J}_{_{X(x)}}\hspace{20pt}\text{for }x\neq z\\
&J_0\times R \hspace{10pt}\text{for }x=z
\end{array}\right\}
\end{equation}
\end{enumerate}

Here $X(x)$ denotes the curve obtained as a quotient of $X_0$ under the identification $x\sim z$.
\end{thm}

\

\begin{proof}
Notice that
\begin{equation}
\widetilde{\Theta_1}|_{\mathcal D_1}\cong \text{Det}~\widetilde{q}^*\mathcal P\otimes (\widetilde{q}^* \mathcal P)|_{\widetilde{r}^{-1}(\Delta)\cap (X_0\times \mathcal D_1)}
\otimes (\widetilde{q}^*\mathcal P)|^{\otimes -{(g-1)}}_{_{p_0\times \mathcal D_1}},
\end{equation}

\begin{equation}
\widetilde{\Theta_1}|_{\mathcal D'_1}\cong \text{Det}~ \widetilde{q}^*\mathcal P\otimes (\widetilde{q}^*\mathcal P)|_{_{z\times \mathcal D'_1}}\otimes (\widetilde{q}^*\mathcal P)|_{_{p_0\times \mathcal D'_1}}^{\otimes -{(g-1)}},
\end{equation}

Also 
\begin{equation}\label{sim1234}
(\widetilde{q}^* \mathcal P)|_{\widetilde{r}^{-1}(\Delta)\cap (X_0\times \mathcal D_1)}\cong (\widetilde{\pi_1})^* \mathcal P,
\end{equation}
where $\widetilde{\pi_1}: \mathcal D_1\rightarrow X_0\times J_0$ is the projection.

Therefore 
$$
\widetilde{\Theta_1}|_{\mathcal D_1}\cong \text{Det}~\widetilde{q}^*\mathcal P\otimes (\widetilde{\pi_1})^* \mathcal P\otimes (\widetilde{p_2}^*\mathcal P_{p_0})^{\otimes -{(g-1)}}, ~\text{and} 
$$
\begin{equation}\label{res123}
\widetilde{\Theta_1}|_{\mathcal D'_1}\cong \text{Det}~\widetilde{q}^*\mathcal P\otimes (\widetilde{p_2}^*\mathcal P_{z})\otimes (\widetilde{p_2}^*\mathcal P_{p_0})^{\otimes -{(g-1)}}.
\end{equation}
Then  
\begin{align*}
\tau^*(\widetilde{\Theta_1}|_{\mathcal D'_1})&\cong \tau^*(\text{Det}~\widetilde{q}^*\mathcal P\otimes (\widetilde{p_2}^*\mathcal P_{z})\otimes (\widetilde{p_2}^*\mathcal P_{p_0})^{\otimes -{(g-1)}}), ~~\eqref{res123}\\
&\cong \tau^*(\text{Det}~ \widetilde{q}^*\mathcal P)\otimes (\widetilde{p_2}^*\mathcal P_{z}\otimes \widetilde{p_1}^*\mathcal O_{X_0}(-z))\otimes ((\widetilde{p_2}^*\mathcal P_{p_0})^{\otimes -{(g-1)}}\otimes \widetilde{p_1}^*\mathcal O_{X_0}(-p_0)^{\otimes-(g-1)}),\\
& (\text{using}~(2)~\text{and}~(3)~\text{in}~\text{remark}~ \eqref{rules})\\
& \cong (\text{Det}~ \mathcal P')\otimes (\widetilde{p_2}^*\mathcal P_{z}\otimes \widetilde{p_1}^*\mathcal O_{X_0}(-z))\otimes ((\widetilde{p_2}^*\mathcal P_{p_0})^{\otimes -{(g-1)}}\otimes \widetilde{p_1}^*\mathcal O_{X_0}(-p_0)^{\otimes-(g-1)}),\\
& (\text{using}~(4)~\text{in remark}~ \eqref{rules})\\
&\cong (\text{Det}~(\widetilde{q}^*\mathcal P)\otimes (\widetilde{q}^*\mathcal P)|_{\widetilde{r}^{-1}(\Delta)}\otimes \widetilde{p_2}^*\mathcal P_{z}^{-1}\otimes \widetilde{p_2}^*\mathcal O_{X_0}(z))\otimes (\widetilde{p_2}^*\mathcal P_{z}\otimes \widetilde{p_1}^*\mathcal O_{X_0}(-z))\\
&\otimes (\widetilde{p_2}^*\mathcal P_{p_0})^{\otimes -{(g-1)}}\otimes \widetilde{p_1}^*\mathcal O_{X_0}(-p_0)^{\otimes-(g-1)},~~(\text{using}~(4) \text{in remark}~\eqref{rules})
\\
&\cong \text{Det}~\widetilde{q}^*\mathcal P\otimes \widetilde{\pi_1}^* \mathcal P\otimes (\widetilde{p_2}^*\mathcal P_{p_0})^{\otimes -{(g-1)}}\otimes \widetilde{p_1}^*\mathcal O_{X_0}(-p_0)^{\otimes-(g-1)},~~\eqref{sim1234}\\
& \cong \widetilde{\Theta_1}|_{\mathcal D_1} \otimes \widetilde{p_1}^*\mathcal O_{X_0}(-p_0)^{\otimes-(g-1)}~~\eqref{res123}
\end{align*}

\

Therefore, over $B^o_1:=X_0\setminus \{p_0\}$, we have 
\begin{equation}
\tau^*(\widetilde{\Theta_1}|_{\mathcal D'_1})\cong \widetilde{\Theta_1}|_{\mathcal D_1}
\end{equation}
Since $\mathcal D_1\cap \mathcal D'_1=\emptyset$, it follows that the line bundle $\widetilde{\Theta_1}$ descends to the algebraic space $\mathcal J_1$. Since $\widetilde{\Theta_1}$ is ample relative to $\widetilde{\mathcal J_1}\rightarrow B^o_1$ and $\widetilde{\mathcal J_1}\rightarrow \mathcal J_1$ is a finite map, the map $\mathcal J_1\rightarrow B^o_1$ is projective. This proves $(1)$.

\

Now let us describe the fibers. For $x\neq z$, the fiber $\widetilde{\mathfrak f_1}^{-1}(x)=\mathbb P(\mathcal P_x\oplus \mathcal P_z)$ contains following two disjoint divisors
\begin{enumerate}
\item $\mathcal D_{1,x}:=$ fiber of $\widetilde{\mathfrak f_1}:\mathcal D_1\rightarrow X_0$  over $x$. 
\item $\mathcal D'_{1,x}:=$ fiber of $\widetilde{\mathfrak f_1}:\mathcal D'_1\rightarrow X_0$  over $x$. 
\end{enumerate}
The restriction of the isomorphism $\tau: \mathcal D_1\rightarrow \mathcal D'_1$ induces an isomorphism $\tau_x: \mathcal D_{1,x}\rightarrow \mathcal D'_{1, x}$. The fiber $\mathcal J_{1,x}$ of the morphism $\mathcal J_1\rightarrow X_0$ at the point $x$ is the quotient of $\mathbb P(\mathcal P_x\oplus \mathcal P_z)$ by the identification $\tau_x$. Using \eqref{Induce1}, it is not difficult to see that there is a family of rank $1$ torsion-free sheaves of degree $0$ over $X(x)$ parametrized by $\mathcal J_{1,x}$. In other words, $\mathfrak f_1^{-1}(x)\cong \overline{J}_{_{X(x)}}$.

For $x=z$, $\widetilde{\mathfrak f_1}^{-1}(z)=\mathbb P(\mathcal P_z\oplus \mathcal P_z)\cong J_0\times \mathbb P^1$. It has two disjoint sections $\mathcal D_{1,z}$ and $\mathcal D'_{1,z}$ which are the fibers of $\widetilde{\mathfrak f_1}: \mathcal D_1\rightarrow X_0$ and $\widetilde{\mathfrak f_1}: \mathcal D'_1\rightarrow X_0$ over the point $z$. The restriction of the isomorphism $\tau$ induces an isomorphism $\mathcal D_{1,z}\rightarrow \mathcal D'_{1,z}$ which maps $L\mapsto L\otimes O_{X_0}(z-z)=L$. Therefore the fiber
$\mathfrak f_1^{-1}(z)\cong J_0\times R$, where $R$ is the rational nodal curve constructed by identifying the two points of the projective line $\mathbb P(\mathbb C\oplus \mathbb C)$ given by the two natural one-dimensional quotients of $\mathbb C\oplus \mathbb C$.
\end{proof}

\

\section{\textbf{Specialization of the compactified Jacobian of an irreducible nodal curve with $k(> 1)$ nodes}}

\

Let us begin by choosing a point $(z_1,\dots,z_k)\in X_0^k$ such that the coordinates $z_i$'s are pairwise distinct. Generalizing the construction in the previous section, we will now construct $\mathcal{J}_k$, a flat family over an open set of $X_0^k$ containing $(z_1,\dots, z_k)$ such that the fiber over a "general" point $(x_1,\dots, x_k)\in X_k$ is isomorphic to $\overline{J}_{_{X(x_1,\dots,x_k)}}$, the compactified Jacobian of the nodal curve $X(x_1,\dots,x_k)$ (see Table \ref{Not1}: Notation and Convention) and the fiber over $(z_1,\dots, z_k)$ is isomorphic to $J_0\times \underbrace{R\times \dots \times R}_{_{k~~\text{times}}}$.

\

\subsection{\textbf{Construction of the total space}}\label{sect 5.1}

\

Let $p_{i,k+1}: X_0^k\times J_0\rightarrow X_0\times J_0$ denote the projection to the product of $i$-th copy of $X_0$ and $J_0$ and
$p_{k+1}:X_0^k\times J_0\rightarrow J_0$ denote the projection to $J_0$.

\

Let $\mathcal P$ be a Poincar\'e bundle over $X_0\times J_0$. For every integer $1\leq i\leq k$, we define a projective bundle 
\[
\mathbb P_i:=\mathbb P(p_{i,k+1}^*\mathcal P\oplus p_{k+1}^* \mathcal P_{z_i})
\]
over $X_0^k\times J_0$. We define
\begin{equation}\label{Def101}
\widetilde{\mathcal J}_k:=\mathbb P_1\times_{_{(X_0^k\times J_{_0})}}\cdots \times_{_{(X_0^k\times J_{_0})}} \mathbb P_k,
\end{equation}

\

\begin{rema}
The variety $\widetilde{\mathcal J_k}$ parametrises tuples $(x_1,\dots, x_k, M, M_{x_i}\oplus M_{z_i}\rightarrow L_i)$, where
\begin{enumerate}
\item $x_1,\dots, x_k$ are points of $X_0$,
\item $M$ is a line bundle of degree $0$ over $X_0$, 
\item $M_{x_i}\oplus M_{z_i}\rightarrow L_i$ is a $1$-dimensional quotient for every $i=1,\dots, k$.
\end{enumerate}
\end{rema}

\

\subsection{\textbf{$k$-pairs of natural divisors on $\widetilde{\mathcal J_k}$}} \label{pairdiv}

\

For each $i\geq1$, we define the following pair of divisors on $\widetilde{\mathcal J_k}$

\begin{equation}\label{Def102}
\mathcal D_i:=\mathbb P_1\times_{(X_0^k\times J_{0})} \cdots \times_{(X_0^k\times J_{0})} \mathbb P_{i-1}\times_{(X_0^k\times J_{0})} \mathbb P(p_{i,k+1}^*\mathcal P)\times_{(X_0^k\times J_{0})} \mathbb P_{i+1}\times_{(X_0^k\times J_{0})} \cdots \times_{(X_0^k\times J_{0})} \mathbb P_k,
\end{equation}
and
\begin{equation}\label{Def103}
\mathcal D'_i:=\mathbb P_1\times_{(X_0^k\times J_{0})} \cdots \times_{(X_0^k\times J_{0})} \mathbb P_{i-1}\times_{(X_0^k\times J_{0})} \mathbb P(p_{k+1}^*\mathcal P_{z_i})\times_{(X_0^k\times J_0)} \mathbb P_{i+1}\times_{(X_0^k\times J_{0})} \cdots \times_{(X_0^k\times J_{0})} \mathbb P_k,
\end{equation}

\

\begin{lema}\label{Dis2}
$\mathcal D_i\cap \mathcal D'_i=\emptyset$ for all $1 \leq i\leq k$.
\end{lema}
\begin{proof}
The two natural sections of $\mathbb P(p_{i,k+1}^*\mathcal P \oplus p_{k+1}^*\mathcal P_{z_i})$ given by the two natural quotients $p_{i,k+1}^*\mathcal P \oplus p_{k+1}^*\mathcal P_{z_i}\rightarrow p_{i,k+1}^*\mathcal P$ and $p_{i,k+1}^*\mathcal P \oplus p_{k+1}^*\mathcal P_{z_i}\rightarrow p_{k+1}^*\mathcal P_{z_i}$ are distinct at every point of $X_0$ (similarly, as in the proof of Lemma \ref{Dis1}), therefore $\mathcal D_i\cap\mathcal D'_i=\emptyset$.
\end{proof}

\

\begin{rema}\label{uni2}
The variety $\mathcal D_i$ parametrises tuples $(x_1,\dots, x_k, M, {\{M_{x_j}\oplus M_{z_j}\rightarrow L_j\}}^k_{j=1})$, where $L_i:=M_{x_i}$ and $L_j$ is any quotient for $j\neq i$. The variety $\mathcal D'_i$ parametrises tuples $(x_1,\dots, x_k, M, {\{M_{x_j}\oplus M_{z_j}\rightarrow L_j\}}^k_{j=1})$, where
$L_i=M_{z_i}$ and $L_j$ is any quotient for $j\neq i$. We see that the varieties $\widetilde{\mathcal J_k}, \mathcal D_i$, and $\mathcal D'_i$ have universal properties because they parametrize these tuples. 
\end{rema}

\

\subsection{\textbf{Twisted isomorphisms $\tau_i: \mathcal D_i\rightarrow \mathcal D'_i$ for $i=1,\dots, k$}}

\

First, let us define a new base which is the following open subset of $X_0^k$.

\

\begin{equation}
B_k:=X_0^{k}\setminus\bigcup_{1\leq i,j\leq k} \big(\Delta_{i,j}\cup \Psi_{i,j} \big), 
\end{equation}
where 
$$\Delta_{i,j}:=\{(x_1,x_2,\dots, x_k):x_i\in X_0\text{ and } x_i=x_j\},$$
$$\Psi_{i,j}:=\{(x_1,x_2,\dots, x_k):x_i\in X_0\text{ and } x_i=z_j\}.$$

\

Let us restrict $\mathcal{J}_k$ over $B_k$. By abuse of notation we will denote $\mathcal{J}_k|_{B_k}$ by $\mathcal{J}_k$. Since the $k$-pairs of irreducible smooth divisors $\{\mathcal D_i, \mathcal D'_i\}_{_{i=1}}^{^{k}}$ on $\widetilde{\mathcal J_k}$  are flat over $X_0^k$, they are also flat over $B_k$. From here onwards, we work over $B_k$ and with the base change of the families. Again, for simplicity of notation, we will denote these base-changed families by the same notations. 

\
For every $i=1, \dots, k$, the twisted isomorphism $\tau_i: \mathcal D_i\to \mathcal D'_i$ is produced using a new Poincar\'e bundle $\mathcal P'_i$ which is constructed by tensoring the old Poincar\'e line bundle $\mathcal P$ over $X_0\times \widetilde{\mathcal J_k}$ by some suitable line bundle. The precise definition is as follows.

\begin{defe}\textbf{(A modified Poincar\'e bundle)}\label{NewPoincare} Consider the following line bundle over $X_0\times \widetilde{\mathcal J_k}$ :

\

\begin{equation}
\mathcal P'_i:= \widetilde{q}^*\mathcal P\otimes \widetilde{r_i}^*\mathcal O(-\Delta) \otimes \widetilde{s}^*\mathcal O(z_i)\end{equation}

\

where the maps $\widetilde{q}, \widetilde{r_i}, \widetilde{s}$ for $i=1, \dots, k$ are composites of projection morphisms described as follows.

\begin{equation}\label{maneref}
\begin{tikzcd}
& X_0\times \widetilde{\mathcal J_k}\arrow{d}{Id\times \widetilde{\pi_k}}\\
& X_0\times X_0^k\times J_0\arrow{dl}{q}\arrow{dd}{r_i}\arrow{dr}{s_i}\\
X_0\times J_0 && X_0 \\
&X_0\times X_0
\end{tikzcd}
\end{equation}

\begin{enumerate}
\item $(x, x_1, \dots, x_k, L)\overset{q}{\mapsto} (x, L),\hspace{3pt} (x, x_1, \dots, x_k, L)\overset{r_i}{\mapsto} (x, x_i)$,
\item $(x, x_1, \dots, x_k, L)\overset{s_i}{\mapsto} x_i$, 
\item  $(x, x_1, \dots, x_k, L)\overset{s}{\mapsto} x$,
\item $\widetilde{\pi_k}: \widetilde{\mathcal J_k}\rightarrow X_0^k\times J_0$ denotes the natural projection map,
\item $\widetilde q:=q\circ
(I\times \widetilde{\pi_k})$, $\widetilde{r_i}:=r_i\circ (I\times \widetilde{\pi_k})$, $\widetilde{s_i}:=s_i\circ (I\times \widetilde{\pi_k})$ and $\widetilde{s}:=s\circ (I\times \widetilde{\pi_k})$. 

\end{enumerate}
\end{defe}
\

\begin{lema}
\begin{enumerate}
\item $\widetilde{r_j}^{-1}(\Delta)\cap (X_0\times \widetilde{\mathcal J_k})\cong \widetilde{\mathcal J_k}$ for all $i,j$.

\item $\widetilde{r_j}^{-1}(\Delta)\cap (X_0\times \mathcal D_i)\cong \mathcal D_i$ for all $i,j$.
\end{enumerate}
\end{lema}

\

\begin{proof}
Similar to the proof of Lemma \ref{iden121}.
\end{proof}

\

Remember from \eqref{Def101} that $\widetilde{\mathcal J}_k:=\mathbb P_1\times_{_{(X_0^k\times J_{_0})}}\cdots \times_{_{(X_0^k\times J_{_0})}} \mathbb P_k$ and that every $\mathbb P_i$ is a projective bundle over $X_0^k\times J_{_0}$. Therefore, by definition, $\mathbb P_i$ has a universal quotient line bundle for every $i=1, \dots, k$. We denote by $\mathcal L_i$ the pullback of the universal quotient line bundle on $\mathbb P_i$ under the projection map $\widetilde{\mathcal J}_k\to \mathbb P_i$ for every $i=1, \dots,  k$.

\

\begin{lema}\label{LemaInvo}
\begin{enumerate}
\item There are $k$ isomorphisms $\tau_i:\mathcal D_i\rightarrow \mathcal D'_i$ for $i=1,\cdots, k$, which can be described as follows:

\begin{equation}\label{Def114}
\tau_i:\mathcal D_i\rightarrow \mathcal D'_i
\left(L, Q_1,\dots, Q_k\right)\mapsto \left(L':=L\otimes \mathcal O(z_i-x_i), Q'_1,\dots, Q'_k\right),
\end{equation}
where $Q'_t:=Q_t$ for $t\neq i$, $Q_{i}$ is the first projection $L_{x_i}\oplus L_{z_i}\rightarrow L_{x_i}$ and $Q'_i$ is the second projection $L'_{x_i}\oplus L'_{z_i}\rightarrow L'_{z_i}$.

\item these automorphisms are compatible i.e., the following diagrams commute
$$
\begin{tikzcd}
\mathcal D_i\cap \mathcal D_j\arrow{r}{\tau_i}\arrow{d}{\tau_j} & \mathcal D'_i\cap \mathcal D_j\arrow{d}{\tau_j}\\
\mathcal D_i\cap \mathcal D'_j\arrow{r}{\tau_i} & \mathcal D'_i\cap \mathcal D'_j
\end{tikzcd}
$$
for every $i\neq j$, and $i,j\in \{1, \dots, k\}$.
\end{enumerate}
\end{lema}

\

\begin{proof}

The restrictions of universal quotients $\{\mathcal L_j\}^k_{j=1}$ on the divisors $\mathcal D_i$  can be expressed as the following collection of $k$ universal quotients. 

\

For $j\neq i$,
\begin{equation}
(\widetilde q^*\mathcal P)|_{\widetilde{r_j}^{-1}(\Delta)\cap (X_0\times \mathcal D_i)}\oplus (\widetilde q^*\mathcal P)|_{z_j\times \mathcal D_i}\rightarrow  \mathcal L_j|_{\mathcal D_i}    \hspace{5pt}\text{ for } j\neq i.
\end{equation}

\

For $j=i$,
\begin{equation}
(\widetilde q^*\mathcal P)|_{\widetilde{r_i}^{-1}(\Delta)\cap (X_0\times \mathcal D_i)}\oplus (\widetilde q^*\mathcal P)|_{z_i\times \mathcal D_i}\rightarrow  (\widetilde q^*\mathcal P)|_{\widetilde{r_i}^{-1}(\Delta)\cap (X_0\times \mathcal D_i)}.
\end{equation}

\

We modify these quotients on $\mathcal D_i$ (using the new Poincar\'e line bundle defined in Definition \ref{NewPoincare}) in the following way.

\

\begin{equation}\label{modi101}
\mathcal P'_i|_{\widetilde{r_j}^{-1}(\Delta)\cap (X_0\times \mathcal D_i)}\oplus \mathcal P'_i|_{z_j\times \mathcal D_i}\rightarrow  \mathcal L'_j   \hspace{5pt} \text{ where } j\neq i,
\end{equation}
and for $j=i$,
\begin{equation}\label{modi102}
\mathcal P'_i|_{\widetilde{r_i}^{-1}(\Delta)\cap (X_0\times \mathcal D_i)}\oplus \mathcal P'_i|_{z_i\times \mathcal D_i}\rightarrow  \mathcal P'_i|_{z_i\times \mathcal D_i},
\end{equation}

\

where $\mathcal L'_j$ is a quotient line bundle defined in the following way. 

\

First, let us denote by $U$ the complement of the divisors $\widetilde{r_i}^{-1}(\Delta)$ and $\widetilde{s}^{-1}(z_j)$ in  $\widetilde{\mathcal J_k}$.
The restrictions of the line bundles $\mathcal P'_i$ and $\widetilde q^*\mathcal P$ on $\mathcal{D}_i$ are naturally isomorphic. To see this notice that any constant function on $U$ defines a section in $\Gamma(U,(\widetilde{r_i}^*\mathcal O(-\Delta) \otimes \widetilde{s}^*\mathcal O(z_i))$,  the set of global sections  of the line bundle $\widetilde{r_i}^*\mathcal O(-\Delta) \otimes \widetilde{s}^*\mathcal O(z_i)$. We fix any such constant section and using it we can identify the restrictions of $\mathcal P'_i$ and $\widetilde q^*\mathcal P$ on $\mathcal D_i$. Therefore, we can define 
\begin{equation}\label{moquo}
\mathcal L'_j:= \mathcal L_j|_{\mathcal D_i}.
\end{equation}

By the universal property (remark \ref{uni2}) of $\mathcal D'_i$, the modified quotients \eqref{modi101} and \eqref{modi102} induce an isomorphism
\begin{equation}
\tau_i: \mathcal D_i\rightarrow \mathcal D'_i
\end{equation}

Since $\mathcal L'_j=\mathcal L_j|_{\mathcal D_i}$ for every $j\neq i$, we have 
\begin{equation}\label{eqa2}
\tau^*_{i}(\mathcal L_{j}|_{ \mathcal D_i '})=\mathcal L_{j}|_{\mathcal D_i}.
 \end{equation}

Therefore, it follows that there are $k$ isomorphisms $\tau_i$ which have the desired properties. The second assertion follows from straightforward checking.  
\end{proof}

\

\subsection{\textbf{Construction of the quotient space and its singularities}} 

\

In this subsection, we construct the quotient space $\mathcal{J}_k$ inductively following \ref{Quo1}. Repeated applications of Theorem \cite[Theorem 3.1]{1} and \cite[Theorem 45 (Gluing of algebraic spaces)]{K} show that $\mathcal{J}_k$ is an algebraic space.

\

Set $\mathcal J_0:=\widetilde{\mathcal J}_k$, $\mathcal D^0_i:=\mathcal D_i$ and $D'^0_i:=\mathcal D'_i$ for every $1\leq i\leq k$. After having defined $\mathcal J_{j-1}$, we define
\begin{equation}\label{DefMain}
\mathcal J_j:=\frac{\mathcal J_{j-1}}{\mathcal D^{j-1}_j\sim_{\tau_j} \mathcal D'^{j-1}_j},
\end{equation}
where $\mathcal D^{j-1}_j$ and $\mathcal D'^{j-1}_j$ are images of $\mathcal D_j$ and $\mathcal D'_j$ in $\mathcal{J}_{j-1}$.
\begin{lema}
$\mathcal D^{j-1}_j\cap \mathcal D'^{j-1}_j=\emptyset$ for every $1\leq j\leq k$.
\end{lema}
\begin{proof}
From lemma \ref{Dis1} it follows that the statement holds for $j=1$. Let us check it for $j=2$ and for this purpose, we can assume that $k=2$. Then the configuration of the divisors $D_1, D'_1, D_2, D'_2$ is the following

\

\begin{center}
\begin{tikzpicture}
\draw[-] (0,-1) -- (6,-1) node[pos=.5,sloped,above] {$\mathcal D_1$};
\draw[-] (0,-4) -- (6,-4) node[pos=.5,sloped,below] {$\mathcal D'_1$};
\draw[-] (1,-5) -- (1, 0) node[pos=.5, below left] {$\mathcal D_2$};
\draw[-] (5,-5) -- (5, 0) node[pos=.5,below right] {$\mathcal D'_2$};
\end{tikzpicture}
\end{center}

\

Let us fix any $(x_1,x_2)\in X_0\times X_0$. It will suffice to check the following
\begin{enumerate}
\item if $(L, q_1, q_2)\in \mathcal D_1\cap \mathcal D_2$ and $(M, p_1, p_2)\in \mathcal D'_1\cap \mathcal D'_2$ then $\tau_1(L, q_1, q_2)\not \cong (M, p_1, p_2)$,
\item if $(L, q_1, q_2)\in \mathcal D_1\cap \mathcal D'_2$ and $(M, p_1, p_2)\in \mathcal D'_1\cap \mathcal D_2$ then $\tau_1(L, q_1, q_2)\not \cong (M, p_1, p_2)$.
\end{enumerate}

It is enough to check one of them because the proofs are the same. Let us check $(1)$. Notice that $q_1:L_{x_1}\oplus L_{z_1}\rightarrow L_{x_1}$ and $q_2: L_{x_2}\oplus L_{z_2}\rightarrow L_{x_2}$ both are first projections. Therefore $\tau_1(L, q_1, q_2):=(L':=L(z_1-x_1), q'_1:L'_{x_1}\oplus L'_{z_1}\rightarrow L'_{z_1}, q'_2: L'_{x_2}\oplus L'_{z_2}\rightarrow L'_{x_2})$, where $q'_1$ is the second projection and $q'_2$ is the first projection. Now notice that $p_1: M_{x_1}\oplus M_{z_1}\rightarrow M_{z_1}$ and $p_2: M_{x_2}\oplus M_{z_2}\rightarrow M_{z_2}$ are both second projections. Since $q'_2$ is the first projcetion and $p_2$ is the second projection, therefore $\tau_1(L, q_1, q_2)\not \cong (M, p_1, p_2)$.

The proof for the general $j$ is similar.
\end{proof}

\

Since the isomorphisms $\tau_i$ commute with the projection onto $B_k$ the morphism $\widetilde{\mathfrak f_k}:\widetilde{\mathcal{J}_k}\rightarrow B_k$ descends to a morphism $\mathfrak f_k:\mathcal{J}_k\rightarrow B_k$.

\

\begin{prop}\label{Sing}
The quotient space $\mathcal J_k$ is an algebraic space and has $k$-th product of normal crossing singularities.
\end{prop}

\

\begin{proof}

From Theorem \ref{Sing1}, it follows that $\mathcal J_1$ has only normal crossing singularities along the image of the divisor $\mathcal D_1$ and that it is smooth elsewhere. Consider the following commutative diagram
\begin{equation}
\begin{tikzcd}
Z:=\mathcal D^1_2\coprod \mathcal D'^1_2\arrow{r}\arrow{d}{\tau_2}& \mathcal J_1\arrow{d}\\
V:= \mathcal D'^1_2 \arrow{r}& \mathcal J_2:=\mathcal J_1/\sim
\end{tikzcd}
\end{equation}
It is enough to check the singularities of $\mathcal J_2$ along the codimension $2$ subspace $\mathcal D'^1_2\cap \mathcal D^1_1$. Let $v'_2\in \mathcal D'^1_2\cap \mathcal D^1_1$ and let $v_2$ and $v'_2$ denote the two preimages under $\tau_2$. Then we have the following

\begin{small}
\begin{equation}
\begin{tikzcd}
0\arrow{r} & \hat{\mathcal O}_{_{\mathcal J_2, v'_2}}\arrow{r}\arrow{d}{=} & \hat{\mathcal O}_{_{\mathcal J_1, v_2}}\oplus \hat{\mathcal O}_{_{\mathcal J_1, v'_2}}\arrow{r}\arrow{d}{\cong} & \hat{\mathcal O}_{_{V, v'_2}}\arrow{r}\arrow{d}{\cong}& 0\\
0\arrow{r} & \hat{\mathcal O}_{_{\mathcal J_2, v'_2}}\arrow{r}& \frac{k[|x_1, x_2|]}{x_1\cdot x_2}[|x_3,x_5,\dots,x_n|]\oplus \frac{k[|x_1, x_2|]}{x_1\cdot x_2}[|x_4,x_5,\dots,x_n|] \arrow{r} & \frac{k[|x_1, x_2|]}{x_1\cdot x_2}[|x_5,\dots,x_n|] \arrow{r} & 0
\end{tikzcd}
\end{equation}
\end{small}

The bottom-right morphism is given by
\begin{align*}
\frac{k[|x_1, x_2|]}{x_1\cdot x_2}[|x_3,x_5,\dots,x_n|]\oplus \frac{k[|x_1, x_2|]}{x_1\cdot x_2}[|x_4,x_5,\dots,x_n|] &\rightarrow \frac{k[|x_1, x_2|]}{x_1\cdot x_2}[|x_5,\dots,x_n|]\\
(f,g)&\mapsto f (\text{mod}~ x_3)-g (\text{mod}~ x_4)
\end{align*}
It follows that $\hat{\mathcal O}_{_{X_2,v'_2}}\cong \frac{k[|x_1, x_2, x_3, x_4|]}{x_1\cdot x_2, x_3\cdot x_4}[|x_5,\dots,x_n|]$.
Therefore the algebraic space $\mathcal J_2$ has the product of two normal crossing singularities along $V$.

At the $i$-th step we have
\begin{equation}
\begin{tikzcd}
Z:=\mathcal D^{i-1}_i\coprod \mathcal D'^{i-1}_i\arrow{r}\arrow{d}{\tau_i}& \mathcal J_{i-1}\arrow{d}\\
V:= \mathcal D'^{i-1}_i \arrow{r}& \mathcal J_i:=\frac{\mathcal J_{i-1}}{\sim}
\end{tikzcd}
\end{equation}

Although it is exactly a similar calculation, we will describe the singularities of $\mathcal J_i$ along the codimension $i$ subspace $\mathcal D'^{i-1}_i\cap \mathcal D^{i-1}_{i-1}\cap \dots \mathcal D^{i-1}_{1}$. Let $v'_i\in \mathcal D'^{i-1}_i\cap \mathcal D^{i-1}_{i-1}\cap \dots \mathcal D^{i-1}_{1}$ and $v_i$ and $v'_i$ are the two pre-images under $\tau_i$. Then we have the following

\begin{equation}
\begin{tikzcd}
0\arrow{r} & \hat{\mathcal O}_{_{\mathcal J_i, v'_i}}\arrow{r}\arrow{dd}{=} & \hat{\mathcal O}_{_{\mathcal J_{i-1}, v_i}}\oplus \hat{\mathcal O}_{_{\mathcal J_{i-1}, v'_i}}\arrow{r}\arrow{dd}{\cong} & \hat{\mathcal O}_{_{V, v'_i}}\arrow{r}\arrow{dd}{\cong}& 0\\ \\
0\arrow{r} & \hat{\mathcal O}_{_{\mathcal J_i, v'_i}}\arrow{r}& R\otimes k[|x_{_{2i-1}}|]\oplus R\otimes k[|x_{_{2i}}|]\arrow{r} &R \arrow{r} & 0
\end{tikzcd}
\end{equation}

where $R:=\frac{k[|x_1, x_2,\dots, x_{2i-3}, x_{2i-2}|]}{x_1\cdot x_2,\dots, x_{2i-3}\cdot x_{2i-2}}[|x_{2i+1},\dots,x_n|]$ and the morphism
\[
R\otimes k[|x_{2i}|]\oplus R\otimes k[|x_{2i+1}|] \rightarrow R
\]
is given by $(f,g)\mapsto
f (\text{mod}~ x_{2i-1})-g (\text{mod}~ x_{2i})$.

Hence $$\hat{\mathcal O}_{_{\mathcal J_i,v'_i}}\cong \frac{k[|x_1, x_2,\dots, x_{2i-1}, x_{2i}|]}{x_1\cdot x_2,\dots, x_{2i-1}\cdot x_{2i}}[|x_{2i+1},\dots,x_n|].$$
Therefore the algebraic space $\mathcal J_i$ has the product of $i$-many normal crossing singularities along $V$.
\end{proof}

\

\subsection{\textbf{Theta bundle on $\widetilde{\mathcal J_k}$ and its relative ampleness}}

\

\begin{defe}
We define a line bundle
\begin{equation}\label{Theta1103}
\widetilde{\Theta_k}:= \text{Det}~~\tilde q^*\mathcal P\otimes (\otimes_{i=1}^k \mathcal L_i)\otimes (\tilde q^*\mathcal P^{\otimes -(g-1)})|_{p_0\times \widetilde{\mathcal J_k}}.
\end{equation}
\end{defe}

\

Notice that the line bundle $\widetilde{\Theta_k}$ is isomorphic to $\widetilde{p_{k+1}}^* \text{Det}~~\mathcal P\otimes (\otimes_{i=1}^k \mathcal L_i)\otimes \widetilde{p_{k+1}}^*\mathcal P^{\otimes -(g-1)}_{_{P_0}}$, where $\widetilde{p_{k+1}}:=p_{k+1}\circ \widetilde \pi_k$ (refer to the diagram \ref{maneref} and the set of notations in the beginning of subsection \ref{sect 5.1} for notations).

\

\begin{prop}
The line bundle $\widetilde{\Theta_k}$ is relatively ample for the morphism $\widetilde{\mathfrak f_k}:\widetilde{\mathcal J_k}\rightarrow X_0^k$.
\end{prop}

\begin{proof}
Let $\overrightarrow{x}:=(x_1,\dots, x_k)\in B_k$ be any point. The restriction of $\widetilde{\Theta_k}$ to the fiber $\widetilde{\mathfrak f_k}^{-1}(\overrightarrow{x})$
\begin{align}
\widetilde{\Theta_k}|_{\widetilde{\mathfrak f_k}^{-1}(\overrightarrow{x})}=&\tilde{p}_{k+1, \overrightarrow{x}}^*\text{Det}~~\mathcal P\otimes \tilde{p}_{k+1, \overrightarrow{x}}^*\mathcal P_{_{P_0}}^{-\otimes (g-1)}\otimes (\otimes_{i=1}^k \mathcal L_i|_{\widetilde{\mathfrak f_k}^{-1}(\overrightarrow{x})})\\
\cong& (\tilde{p}_{k+1, \overrightarrow{x}}^*\text{Det}~~\mathcal P\otimes \tilde{p}_{k+1, \overrightarrow{x}}^*\mathcal P_{_{P_0}}^{-\otimes (g-k-1)})\otimes (\otimes_{i=1}^k  \mathcal L_i|_{\widetilde{\mathfrak f_k}^{-1}(\overrightarrow{x})})\otimes \tilde{p}_{k+1, \overrightarrow{x}}^*\mathcal P_{_{P_0}}^{k-1}
\\ \label{hello123} \cong& \tilde{p}_{k+1, \overrightarrow{x}}^*\Theta_0\otimes (\otimes_{_{i=1}}^k \mathcal O_{\mathbb P_i(\overrightarrow{x})}(1))\otimes \tilde{p}_{k+1, \overrightarrow{x}}\mathcal P_{P_0}^{k-1},
\end{align}
where 
\begin{enumerate}
\item $\mathbb P_i(\overrightarrow{x})$ denotes the projective bundle $\mathbb P(\mathcal P_{x_i}\oplus \mathcal P_{z_i})$ over $J_0$,
\item $\mathcal O_{\mathbb P_i(\overrightarrow{x})}(1)$ denotes the pullback of the tautological bundle of the projective bundle $\mathbb P_i(\overrightarrow{x})$ by the natural projection morphism $\mathbb P_1(\overrightarrow{x})\times_{J_0}\times\cdots\times_{J_0} \mathbb P_k(\overrightarrow{x})\rightarrow \mathbb P_i(\overrightarrow{x})$, and 
\item $\tilde{p}_{k+1,\overrightarrow{x}}$ denotes the natural projection $\widetilde{\mathfrak f_k}^{-1}(\overrightarrow{x})\rightarrow J_0$. 
\end{enumerate}
Notice that we have the isomorphism \eqref{hello123}  because $\mathcal O_{\mathbb P_i(\overrightarrow{x})}(1)\cong \mathcal L_i|_{\widetilde{\mathfrak{f}_k}^{-1}(\overrightarrow{x})}$ for every $i=1,\dots, k$.
Now consider the Segre-embedding
\begin{equation}
\mathbb P_1(\overrightarrow{x})\times\cdots \times \mathbb P_k(\overrightarrow{x})\hookrightarrow  \mathbb P(\otimes_{i=1}^k (\mathcal P_{x_i}\oplus \mathcal P_{z_i})
\end{equation}
Notice that

\begin{equation}
(\tilde{p}_{k+1,\overrightarrow{x}})_*(\tilde{p}_{k+1,\overrightarrow{x}}^*\Theta_0\otimes \mathcal O(1))\cong \Theta_0\otimes (\otimes_{i=1}^k (\mathcal P_{x_i}\oplus \mathcal P_{z_i}))
\end{equation}

Now every direct summand of $\Theta_0\otimes (\otimes_{i=1}^k (\mathcal P_{x_i}\oplus \mathcal P_{z_i}))$ is equal to $\Theta_0\otimes (\otimes_{i=1}^k \mathcal P_{p_i})$, for some  $p_i\in \{x_i,z_i\}$ and $i\in \{1,\dots ,k\}$. So, since every direct summand is ample, the vector bundle is ample. Also, notice that the line bundle $\mathcal P_{p_0}$ is algebraically equivalent to the trivial line bundle. Therefore the line bundle ${\tilde{p}_{k+1,\overrightarrow{x}}}^*\Theta_0\otimes \mathcal O(1)$ on $\mathbb P(\otimes_{i=1}^k (\mathcal P_{x_i}\oplus \mathcal P_{z_i}))$ is ample. Hence its pullback $\widetilde{\Theta}_k$ is an ample line bundle on $\mathbb P(\mathcal P_{x_1}\oplus \mathcal P_{z_1})\times_{J_0} \cdots \times_{J_0} \mathbb P(\mathcal P_{x_k}\oplus \mathcal P_{z_k})$. Since the morphism $\widetilde{\mathcal J_k}\rightarrow X_0^k$ is projective and the line bundle $\widetilde{\Theta_k}$ is ample on every fiber of the morphism, $\widetilde{\Theta_k}$ is relatively ample.
\end{proof}

\

\subsection{\textbf{Descent of the line bundle $\widetilde{\Theta_k}$}}

\

\begin{lema}\label{relations123}
\begin{enumerate}
\item $\tau_i^*(\widetilde{p_{k+1}}^*\mathcal P_{x})\cong \widetilde{p_{k+1}}^*\mathcal P_{x}\otimes \tilde{s_i}^*\mathcal O_{X_0}(-x)$, for any point $x\in X_0$,
\item $\tau_i^*((\text{Det}~~\tilde q^*\mathcal P)|_{\mathcal D'_i})\cong \widetilde{p_{k+1}}^*(\text{Det}~~\mathcal P)\otimes (\widetilde q^*\mathcal P)|_{\widetilde{r_i}^{-1}(\Delta)\cap (X_0\times \mathcal D_i)}\otimes (\widetilde q^*\mathcal P)^{-1}|_{z_i\times \mathcal D_i}\otimes \tilde s_i^* \mathcal O_{X_0}(z_i)$
\end{enumerate}
\end{lema}

\

\begin{proof}
Similar to the proof of remark \ref{rules}.
\end{proof}

\

\begin{defe}\label{special1102}
We define a smaller open set 
\begin{equation}
B^o_k:=B_k\setminus \{(x_1,\dots, x_k)| x_i=p_0~~\text{for some}~~i\in \{1,\dots, k\}\} 
\end{equation}

\

We restrict the family $\mathfrak f_k: \mathcal J_k\to B_k$ to the smaller open subset $B^o_k$. By abuse of notation, we denote $\mathcal J_k|_{B^o_k}$ by $\mathcal J_k$. Finally, the family $\mathfrak f_k: \mathcal J_k\to B^o_k$ is our \textbf{desired family of specializations of $\overline{J_k}$}. 
\end{defe}

\

\begin{thm}\label{MT11}
\begin{enumerate}
\item The morphism $\mathfrak f_k: \mathcal{J}_k\rightarrow B^o_k$ is  projective.
\item The fibers of the morphism $\mathfrak f_k$ can be described as follows.
\begin{equation}
\mathfrak f_k^{-1}(x_1,\dots,x_k)\cong \left\{
\begin{array}{@{}ll@{}}
&\overline{J}_{_{X(x_1,x_2,\dots,x_k)}}\hspace{33pt}\text{if}~~x_i\neq z_i~~\text{for all}~~1 \leq i\leq k\\
&J_0\times R^k \hspace{45pt}\text{if}~~x_i=z_i~~\text{for all}~~1 \leq i\leq k
\end{array}\right\}
\end{equation}
\end{enumerate}

\end{thm}
\begin{proof}
First, we claim that 

\

\underline{Claim:} the line bundle $\widetilde{\Theta_k}$ is invariant under the isomorphisms $\tau_i$ for all $i=1,\dots, k$.

\

Assuming the claim, we see that $\widetilde{\Theta_k}$ descends at each of the $k$-steps of the quotient construction. Let us denote the descended line bundle on $\mathcal J_k$ by $\Theta_k$. Since $\widetilde{\Theta_k}$ is a relatively ample line bundle for the proper morphism $\widetilde{\mathcal{J}_k}\rightarrow B^o_k$, and $\nu_k: \widetilde{\mathcal J_k}\rightarrow \mathcal J_k$ is a finite morphism, the descended line bundle $\Theta_k$ is also relatively ample for the morphism $\mathfrak f_k: \mathcal J_k\rightarrow B^o_k$. Therefore $\mathfrak f_k$ is a projective morphism. 

\

\underline{proof of the claim:} The proof of the claim is similar to the proof of Theorem \ref{Descent}. For every $i=1,\dots, k$, we have
\begin{align*}
\widetilde \Theta_k|_{\mathcal D_i}
&\cong (\text{Det}~~\tilde q^*\mathcal P)|_{\mathcal D_i}\otimes (\otimes_{j\neq i} \mathcal L_j|_{\mathcal D_i})\otimes \mathcal L_i|_{\mathcal D_i}\otimes (\tilde q^*\mathcal P^{\otimes -(g-1)})|_{p_0\times \mathcal D_i}\\
&\cong \widetilde{p_{k+1}}^*(\text{Det}~~\mathcal P)\otimes (\otimes_{j\neq i} \mathcal L_j|_{\mathcal D_i})\otimes (\widetilde q^*\mathcal P)|_{\widetilde{r_i}^{-1}(\Delta)\cap (X_0\times \mathcal D_i)}\otimes \widetilde{p_{k+1}}^*(\mathcal P_{p_0}^{\otimes -(g-1)})\\
\end{align*}
and 
\begin{align*}
\widetilde \Theta_k|_{\mathcal D'_i}
&\cong (\text{Det}~~\tilde q^*\mathcal P)|_{\mathcal D'_i}\otimes (\otimes_{j\neq i} \mathcal L_j|_{\mathcal D'_i})\otimes \mathcal L_i|_{\mathcal D'_i}\otimes (\tilde q^*\mathcal P^{\otimes -(g-1)})|_{p_0\times \mathcal D'_i}\\
&\cong \widetilde{p_{k+1}}^*(\text{Det}~~\mathcal P)\otimes (\otimes_{j\neq i} \mathcal L_j|_{\mathcal D'_i})\otimes (\widetilde q^*\mathcal P)|_{z_i\times \mathcal D'_i}\otimes \widetilde{p_{k+1}}^*(\mathcal P_{p_0}^{\otimes -(g-1)})\\
\end{align*}

Therefore,
\begin{align*}
\tau_i^*(\widetilde \Theta_k|_{\mathcal D'_i})
&\cong \tau_i^*((\text{Det}~~\tilde q^*\mathcal P)|_{\mathcal D'_i})\otimes \tau_i^*(\otimes_{j\neq i} \mathcal L_j|_{\mathcal D'_i})\otimes \tau_i^*((\widetilde q^*\mathcal P)|_{z_i\times \mathcal D'_i})\otimes \tau_i^*(\widetilde{p_{k+1}}^*(\mathcal P_{p_0}^{\otimes -(g-1)}))\\
&\cong \tau_i^*((\text{Det}~~\tilde q^*\mathcal P)|_{\mathcal D'_i})\otimes \otimes_{j\neq i} \mathcal L_j|_{\mathcal D_i}\otimes (\widetilde q^*\mathcal P)|_{z_i\times \mathcal D_i}\otimes \tilde s_i^* \mathcal O_{X_0}(-z_i)\otimes\\
& \otimes \widetilde{p_{k+1}}^*(\mathcal P_{p_0}^{\otimes -(g-1)}) \otimes \tilde s_i^* \mathcal O_{X_0}(-p_0)^{\otimes -(g-1)}, ~~(\eqref{eqa2}~\text{and}~\eqref{relations123})\\
&\cong \widetilde{p_{k+1}}^*(\text{Det}~~\mathcal P)\otimes (\widetilde q^*\mathcal P)|_{\widetilde{r_i}^{-1}(\Delta)\cap (X_0\times \mathcal D_i)}\otimes (\widetilde q^*\mathcal P)^{-1}|_{z_i\times \mathcal D_i}\otimes \tilde s_i^* \mathcal O_{X_0}(z_i) \otimes(\otimes_{j\neq i} \mathcal L_j|_{\mathcal D_i})\\
&\otimes (\widetilde q^*\mathcal P)|_{z_i\times \mathcal D_i}\otimes \tilde s_i^* \mathcal O_{X_0}(-z_i) \otimes \widetilde{p_{k+1}}^*(\mathcal P_{p_0}^{\otimes -(g-1)})\otimes \tilde s_i^* \mathcal O_{X_0}(-p_0)^{\otimes -(g-1)}\\
& \cong \widetilde{p_{k+1}}^*(\text{Det}~~\mathcal P)\otimes (\otimes_{j\neq i} \mathcal L_j|_{\mathcal D_i})\otimes (\widetilde q^*\mathcal P)|_{\widetilde{r_i}^{-1}(\Delta)\cap (X_0\times \mathcal D_i)}\otimes \widetilde{p_{k+1}}^*(\mathcal P_{p_0}^{\otimes -(g-1)}) \otimes\\
& \otimes \tilde s_i^* \mathcal O_{X_0}(-p_0)^{\otimes -(g-1)}\\
& \cong \widetilde{\Theta_k}|_{\mathcal D_i}\otimes \tilde s_i^* \mathcal O_{X_0}(-p_0)^{\otimes -(g-1)}
\end{align*}

Therefore, over $B^o_k$, we have 
\begin{equation}
\tau_i^*(\widetilde{\Theta_k}|_{\mathcal D'_i})\cong \widetilde{\Theta_k}|_{\mathcal D_i} 
\end{equation}

The proof of the second statement is similar to the proof of the second statement of Theorem \ref{Descent}.
\end{proof}

\section{\textbf{Local triviality of the family of $\mathcal J_k$ over $B^o_k$}}

The main theme of this section is to prove that  $\mathcal{J}_k$ is a topological fiber bundle over $B^o_k$. As $\mathfrak f_k$ is not a smooth map, one can not use the Ehressman fibration Theorem. Instead, we apply the \emph{The first Isotopy lemma} of Thom to conclude that $\mathfrak f_k$ is a locally trivial fibration. To do that, we need to construct a stratification $\mathbb{S}$ of $\mathcal{\J}_k$ which satisfies \emph{Whitney's conditions} and also such that the restriction of the map $\mathfrak f_k: \mathcal{J}_k\to B^o_k$ to each stratum is a submersion.

\

Let $M$ be a smooth manifold and $N$ be a closed subset of $M$. A collection $\mathbb{S}:=\{X_\alpha, \alpha\in I ~|~ X_\alpha \text{ are locally closed submanifold of } M\}$ is said to be a \textbf{stratification} of $N$ if $N\cong \underset{\alpha\in I}{\bigsqcup}X_{\alpha}$ and $\bar{X}_{\alpha}\setminus X_{\alpha}=\bigsqcup X_{\beta}$, for some $\beta\in I$ and $\beta\neq\alpha$.

\

\subsection{\textbf{Whitney's conditions}}\label{Wh}
A stratification $\mathbb{S}$ of $N$ is said to be a \textbf{Whitney stratification} if $\mathbb{S}$ is locally finite and satisfies the following conditions at every point $x\in N$. Let us choose a pair $(X_{\alpha}, X_{\beta})$ such that $X_{\beta}\subset \bar{X}_{\alpha}$ and $x\in X_{\beta}$.
\begin{enumerate}
\item{\textsf{Condition (a):}}
We say that the pair $(X_{\alpha}, X_{\beta})$ satisfies the \textbf{Whitney's condition $(a)$} at $x$ if 
for any  sequences $\{x_n\}\subset X_{\alpha}$ such that $\{x_n\}$ converges to $x$, the sequence  $\{T_{x_n}X_{\alpha}\}$ of tangent planes of $X_{\alpha}$ at $x_n$ converges to a plane $\mathcal{T}:=\lim T_{x_n}X_{\alpha}\subset T_{x}M$ of $\dim(X_{\alpha})$ and   $T_{x}X_{\beta}\subset \mathcal{T}$ where $T_{x}X_{\beta}$ is the tangent plane of $X_{\beta}$ at $x$. 
\item{\textsf{Condition (b):}}
The pair $(X_{\alpha}, X_{\beta})$ satisfies the \textbf{Whitney's condition $(b)$} at $x$ if for any 
sequences $\{x_n\}\subset X_{\alpha}$, $\{y_n\}\subset X_{\beta}$ converging to $x$, then $\mathcal{T}\supset \tau$, the limit of the secants joining $x_n$ and $y_n$, $\tau:=\lim \overline{x_ny_n}$. 
\end{enumerate}

\

Any stratification that satisfies the above conditions is called a Whitney stratification. A stratification that satisfies the Condition $(b)$ of Whitney will also satisfy the Condition $(a)$ \cite[Lemma 2.2]{9}. 

\

The following lemma is well-known, and we leave the proof to the reader.

\

\begin{lema}\label{eqn1}
Let $X$ and $Y$ be two varieties equipped with Whitney stratifications $A$ and $B$, respectively. Then the product stratification is also a Whitney stratification on $X\times Y$.
\end{lema}

\

\subsection{\textbf{Stratification by successive singular loci}}

In this subsection, we will describe a natural Whitney stratification on $\mathcal J_k$. The stratification is constructed as follows. First, we show that $\widetilde{\mathcal{J}}_k$ has a natural stratification. Then, the image of these stratifications under the map $\nu_k: \widetilde{\mathcal{J}}_k\to \mathcal J_k$ defines a stratification on $\mathcal J_k$. Roughly speaking, the stratification is given by the loci of torsion-free sheaves, which are not locally free at a given subset of the nodes. The precise description is as follows. 

\

Given a subset $\{i_1, \dots, i_r\}\subset [1,\dots,k]$ with $1\leq i_1<\dots<i_r\leq k$ and a map $\phi: \{i_1, \dots, i_r\}\rightarrow \{1,2\}$ we define a stratification of $\mathcal{J}_k$ as follows.
Recall $$\widetilde{\mathcal{J}}_k:=\mathbb P(p_{1,k+1}^*\mathcal P\oplus p_{k+1}^* \mathcal P_{z_1})\times_{_{X_0^k\times J_0}} \cdots \times_{_{X_0^k\times J_0}} \mathbb P(p_{k,k+1}^*\mathcal P\oplus p_{k+1}^* \mathcal P_{z_k}).$$
Consider the subvariety

$$
\widetilde{W}^{\phi}_{i_1,\dots, i_r}:=\mathbb P_1\times_{_{X_0^k\times J_0}} \dots \times_{_{X_0^k\times J_0}} \mathbb P_{i_1-1}\times_{_{X_0^k\times J_0}} \widetilde{W}^{\phi(i_1)}_{i_1}\times_{_{X_0^k\times J_0}} \mathbb P_{i_1+1}\times_{_{X_0^k\times J_0}} \dots
$$
$$
\cdots \times_{_{X_0^k\times J_0}} \mathbb P_{i_r-1}\times_{_{X_0^k\times J_0}} \widetilde{W}^{\phi(i_r)}_{i_r}\times_{_{X_0^k\times J_0}} \mathbb P_{r_1+1}\times_{_{X_0^k\times J_0}} \dots\times_{_{X_0^k\times J_0}} \mathbb P_k,
$$

where
$$
\widetilde{W}^{\phi(i)}_i:=\mathbb P(p^*_{i,k+1} \mathcal P) ~~~~\hspace{2em}\text{if}~~~~ \phi(i)=1
$$
and
$$
\widetilde{W}^{\phi(i)}_i:=\mathbb P(p^*_{k+1} \mathcal P_{z_i}) ~~~~\hspace{2em}\text{if}~~~~ \phi(i)=2.
$$
Define $$\widetilde{S}_r:=\bigcup_{1\leq i_1<\dots<i_r\leq k, \text{ }\phi } \widetilde{W}^{\phi}_{i_1,\dots, i_r}.$$
Set $S_r:=\nu_k(\widetilde{S}_r)$, where $\nu_k: \widetilde{\mathcal J}_k\rightarrow \mathcal{J}_k$. It is easy to see that 

$$
S_k\subset S_{k-1}\subset \dots \subset S_1 \subset S_0:=\mathcal{J}_k.$$
It follows that $S_0=\coprod^{i=k}_{i=0} (S_{i}\setminus S_{i+1})$, where $S_{k+1}=\varnothing$, the empty set. In particular, we obtain $\mathbb{S}:=\{S_i\setminus S_{i+1}:0\leq i\leq k\}$,
a stratification of $\mathcal{J}_{k}$.

\

\begin{rema}\label{it2}
From Proposition \ref{Sing}, it follows that the singular locus of every connected component of $S_i$ lies in a unique connected component of $S_{i+1}$ for every $i=0, \dots, k$.
In fact, along $(S_i\setminus S_{i+1})$ the variety $S_0$ has a product of exactly $i$-many normal crossing singularities.\end{rema}

\

\begin{thm}
The stratification $\mathbb{S}$ on $\mathcal{J}_k$ is a Whitney stratification.
\end{thm}

\

\begin{proof}
Since the Whitney conditions are local properties, it is enough to prove it for small analytic neighbourhoods of every point. It is also enough to prove for small analytic neighbourhood of every point $p\in S_k$, because the proof for other points will be similar. For any point $p\in S_k$, there exists a local analytic neighborhood $U_p$ of $p$ which is homeomorphic to
$ X_0\times X_1\times \dots \times X_k$ where $X_0\cong \mathbb{A}^{m}$, $X_i\cong \Spec(\frac{k[[x,y]]}{xy})$ and $m:=\dim(\mathcal{J}_k)-k$.

It is enough to prove that the restriction of the stratification $\mathbb{S}:=\{S_{i}\setminus S_{i+1}:0\leq i\leq k\}$ to $U_p$ is a Whitney stratification on $U_p$. To do that we show that the restriction of $\mathbb{S}$ to $U_p$ is the product of a Whitney stratification of each $X_i$. Then using Lemma \ref{eqn1}, it follows that $\mathbb{S}$ is a Whitney stratification.

Let us consider the case when $X_0$ is a point. We have $U_p\cong X_1\times \dots \times X_k$.
The product stratification on $X_1\times \dots \times X_k$ is the following. For each $i\geq1$, the Whitney stratification on $X_i$ is $$X_i=\left(X_i\setminus 0_i\right)\amalg 0_i,$$
where $0_i$ is the only singular point of $X_i$.
Then the zero dimensional strata on $X_1\times \dots \times X_k$ is
$$\mathbb{T}_0:=(0_1,\dots ,0_k).$$
Let us define $T_1^i:= 0_1\times 0_2\times\dots\times \left(X_i\setminus 0_i\right)\times 0_{i+1}\times\dots \times0_k.$
The one-dimensional strata is
$$\mathbb{T}_1:=\bigcup_{1\leq i\leq k}T_1^i.$$
Similarly, a typical $j$-dimensional stratum is
$$T_j^{i_1,\dots, i_j}:= 0_1\times\dots\times\left( X_{i_1}\setminus 0_{i_1}\right)\times 0_{i+1}\times\dots\times\left( X_{i_j}\setminus 0_{i_j}\right)\times\dots\times 0_k,$$
and the $j$-dimensional strata is
$$\mathbb{T}_j:=\bigcup_{1\leq i_1<\dots<i_r\leq k }T_j^{i_1,\dots, i_j}.$$
In particular for $j=k$, the $k$-dimensional strata is the following
$$\mathbb{T}_k=\left(X_1\setminus 0_1\right)\times \dots \times \left(X_k\setminus 0_k\right).$$
The product stratification $\mathbb{T}=\{\mathbb{T}_{j}:1\leq j\leq k\}$ can also be expressed as
$$\mathbb{T}_j=\tilde{T}_j\setminus \tilde{T}_{j-1},$$
%
%
where $$\tilde{\mathbb{T}}_j:=\bigcup_{1\leq i_1<\dots<i_r\leq k }\tilde{T}_j^{i_1,\dots, i_j}$$
and $$\tilde{T}_j^{i_1,\dots, i_j}=0_1\times\dots\times X_{i_1}\times 0_{i+1}\times\dots\times X_{i_j}\times\dots \times0_k.$$

Evidently, $\tilde T_{k-j}$ is the locus of points at which $U_p$ has the product of exactly $j$-many normal crossings singularities. Therefore from the Remark \ref{it2} it follows that ${{S}_{j}}|_{_{U_p}}=\tilde{T}_{k-j}$. Hence $\mathbb{S}$ is a Whitney stratification. The general case will follow by replacing $\mathbb S:=\{S_i\setminus S_{i+1}: 0\leq i\leq k\}$ by $\mathbb{A}^{m}\times \mathbb S:=\{\mathbb A^m\times (S_i\setminus S_{i+1}) : 0\leq i\leq k\}$.

\end{proof}

\

\begin{thm}\label{maint}\label{MixH1}
\begin{enumerate}
\item The morphism $ \mathfrak f_k:\mathcal J_k\rightarrow B^o_k$ is topologically locally trivial.
\item $\mathcal{R}^i \mathfrak f_{k*}\mathbb{Q}$ forms a variation of mixed Hodge structures over $B^o_k$.
\end{enumerate}
\end{thm}

\

\begin{proof}
There is a relatively ample line bundle $\Theta_k$ on the projective variety $\mathcal{J}_k$. We can replace $\Theta_k$ by its sufficiently large power so that it is relatively very ample. Therefore we have an embedding:
\begin{equation}
\begin{tikzcd}
\mathcal{J}_k\arrow[hook]{rr}\arrow{dr}&& \mathbb P(H^0(\Theta_k))\arrow{dl}\\
& B^o_k
\end{tikzcd}
\end{equation}
The morphism $\mathbb P(H^0(\Theta_k))\rightarrow B^o_k$ is a submersion and $\mathcal{J}_k$ is a closed subset of $\mathbb P(H^0(\Theta_k))$ which has a Whitney stratification given by $\mathcal{J}_k=\cup^{i=k}_{i=0} (S_{i}\setminus S_{i+1})$ such that the projection from every strata $S_i\setminus S_{i+1}\rightarrow B^o_k$ is a submersion. Therefore from Thom's first isotropy theorem \cite[Proposition 11.1]{9} it follows that $\mathcal{J}_k\rightarrow B^o_k$ is topologically locally trivial. This proves (1).

\

By (1), $\mathfrak f_{k}$ is topologically locally-trivial. Hence $\mathcal{R}^i \mathfrak f_{k*}\mathbb{Q}$ is a locally constant sheaf of finite type over $B^o_k$  for all $i$. Since $B^o_k$ is nonsingular, $\mathcal{R}^i \mathfrak f_{k*}\mathbb{Q}$ forms a variation of mixed Hodge structures over $B^o_k$ with a canonical choice of $\{\mathcal{W}_{n}\}$ and $\{\mathcal{F}^p\}$ \cite[Proposition 8.1.16]{2}. This proves (2).
\end{proof}

\

\section{\textbf{Applications: Betti numbers and mixed Hodge numbers of the cohomologies of a compactified Jacobian}}

As before, let $k$ be a positive integer. Let $X_k$ denote any irreducible nodal curve of arithmetic genus $g$. Let us denote its normalization by $q_k: X_0\rightarrow X_k$. Let us denote the nodes of $X_k$ by $\{y_1,\dots, y_k\}$ and the inverse image of the node $y_i$ under the normalization map by $\{x_i, z_i\}$ for every $i=1,\dots, k$. We fix such a nodal curve $X_k$. We denote its compactified Jacobian by $\overline{J}_k$ and its normalization by $\widetilde{J_k}$.

In \cite[Section 5]{3}, Bhosle and Parameswaran computed the Betti numbers of $\bar{J}_k$ by comparing the Betti numbers with that of the normalization of $\bar{J}_k$ and using induction on the genus of the nodal curve. Here, we discuss a different way to compute the Betti numbers using the family $\mathcal J_k$. We also compute the mixed Hodge numbers of $\bar{J}_k$.

\

\begin{thm}

\begin{enumerate}

\item Then $i$-th betti number of $\bar{J}_k$
\begin{equation}\label{we123}
h^{i}(\bar{J}_k)=h^{i}\left(J_0\times R^k\right)=\sum_{0\leq l\leq \tt{min}\{i,2k\}}\binom{2(g-k)}{i-l}.\sum_{\frac{1}{2}\leq j\leq \tt{min}\{l,k\}}\binom kj. \binom j{2j-l}.
\end{equation}

\item The dimension of $gr_{l}^{W}\left(H^{i}(\bar{J}_k)\right)$ is
\begin{equation}\label{we1234}
\dim_{\mathbb{Q}}\gr_{l}^{W}\left(H^{i}(\bar{J}_k)\right)= \sum_{0\leq t\leq l, (l-t)~ \text{is even}} \binom {2(g-k)}{t}.\binom {k}{i-\frac{l-t}{2}}. \binom{i-\frac{l-t}{2}}{i-l+t} 
\end{equation}
and
\item
For $p$, $q\geq0$ such $p+q=l$, the dimension of 
\[\dim_{\mathbb{C}}\gr_F^p\gr^p_{\bar{F}}\left(\gr_{l}^{W}\left(H^{i}(\bar{J}_k)\right)\right)=\sum_{0\leq t\leq l, (l-t)~ \text{is even}} \binom {g-k}{p-\frac{l-t}{2}}\binom{g-k}{q-\frac{l-t}{2}}\binom {k}{i-\frac{l-t}{2}}. \binom{i-\frac{l-t}{2}}{i-l+t}\]
 \end{enumerate}
\end{thm}

\

\begin{proof}
\underline{\textsf{proof of (1).}} Since the family $\mathfrak f_k: \mathcal J_k\rightarrow B^o_k$ constructed in section $4$ is topologically locally trivial (by Theorem \ref{maint}), the fiber over $(x_1,\dots, x_k)$ is homeomorphic to the fiber over $(z_1,\dots, z_k)$. Thus their Betti numbers agree i.e., 
\begin{equation}
h^{i}(\bar{J}_k)=h^{i}\left(J_0\times R^k\right).
\end{equation}
Now consider the Kunneth decomposition
\begin{equation}\label{key34}
H^{i}(R^k)=\bigoplus_{0\leq t\leq j\leq k }
\left(\bigotimes^{k-j}H^{0}(R)\bigotimes^{t} H^{1}(R)\bigotimes^{j-t} H^{2}(R)\right),
\end{equation}
where $2j-t=i$.
Since each of the Kunneth components are one dimensional,
\begin{align*}
h^{i}(R^k)=&\sum_{0\leq t\leq j\leq k }\binom k{k-j}.\binom jt=\sum_{\frac{i}{2}\leq j\leq \tt{min}\{i,k\}}\binom kj.\binom j{2j-i}\notag\\
=&\sum_{\frac{i}{2}\leq j\leq \tt{min}\{i,k\}}\binom kj.\binom j{2j-i}.\label{eq11}
\end{align*}
\begin{align*}
h^{i}\left(J_0\times R^k\right)=&\sum_{0\leq l\leq \tt{min}\{i,2k\}}h^{i-l}(J_0).h^{l}(R^k)\\
=&\sum_{0\leq l\leq \tt{min}\{i,2k\}}\binom {2(g-k)}{i-l}.\sum_{\frac{l}{2}\leq j\leq \tt{min}\{l,k\}}\binom kj. \binom j{2j-l}.\hspace{30pt}\left(\text{by }\eqref{key34}\right)
\end{align*}
Hence the proof of \eqref{we123} follows. 

\

\underline{\textsf{proof of (2) and (3).}}
From Theorem (\ref{MixH1}), $\mathcal{R}^i\mathfrak f_{k*}\mathbb{Q}$ forms a VMHS. Thus for each $j\geq0$, $\gr_{j}^{\mathcal{W}}\left(\mathcal{R}^i\mathfrak f_{k*}\mathbb{Q}\right)$ forms a canonical variation of Hodge structures. In particular the dimension  and Hodge numbers of $\gr_{j}^{W}\left(H^{i}(J_k)\right)$  and $\gr_{j}^{W}\left(H^{i}(J(X_0)\times R^k)\right)$ are equal. 

For the rational nodal curve $R$, the cohomology $H^{2}(R)$ has pure weight $2$ of type $(1,1)$
and $H^{1}(R)$ and $H^{0}(R)$ have weight $0$ of type $(0,0)$.  Therefore, the weight of each summand in \eqref{key34} is $2(j-t)$ and type $(j-t,j-t)$. In particular, each summand is isomorphic to the Hodge-Tate structure $\mathbb{Q}(t-j)$. 
 
Hence, for any $l\geq0$, 
$$
\dim_{\mathbb{Q}}\gr_{2l+1}^{W}\left(H^{i}\left(R^k\right)\right)=0.
)$$ 
and $\gr_{2l}^{W}\left(H^{i}\left(R^k\right)\right)$ is isomorphic to direct sum of $\mathbb{Q}(-l)$ as a mixed Hodge structures.
Thus for all $i\geq0$, $H^{i}\left(R^k\right)$ has a mixed Hodge-Tate structure.
In order to compute the dimension of $\gr_{2l}^{W}\left(H^{i}\left(R^k\right)\right)$, using \eqref{key34}, one obtains 
\begin{equation}\label{we12}
j-t=l
\end{equation}
\begin{equation}\label{we14}
2j-t=i .
\end{equation}
Solving \eqref{we12} and \eqref{we14} we have $t=i-2l$ and $j=i-l$.
Therefore
\begin{equation}\label{we15}
\dim_{\mathbb{Q}}\gr_{2l}^{W}\left(H^{i}\left(R^k\right)\right)=\binom k{i-l}. \text{ }
\binom {i-l}{i-2l}
\end{equation}
Now consider

\begin{equation}\label{it11}
    \dim_{\mathbb{Q}}\gr_{l}^{W}\left(H^{i}\left(J_0\times R^k\right)\right)=\sum_{0\leq t\leq l}h^{t}(J_0).\dim_{\mathbb{Q}}\left(\
gr^{W}_{l-t}H^{i-t}\left(R^k\right)\right).
\end{equation}
Since $H^{i}\left(R^k\right)$ is a mixed Hodge-Tate structure, $\gr^{W}_{l-t}H^{i-t}\left(R^k\right)=0$
if and only if $l\neq t(\text{ mod } 2)$. Then from \eqref{it11}, one has

\begin{align*}
\dim_{\mathbb{Q}}\gr_{l}^{W}\left(H^{i}\left(J_0\times R^k\right)\right)=&\sum_{0\leq t\leq l}h^{t}(J_0).\dim_{\mathbb{Q}}\left(\gr^{W}_{l-t}H^{i-t}\left(R^k\right)\right)\notag\\
=&\sum_{0\leq t\leq l, (l-t)~ \text{is even}}\binom {2(g-k)}{t}.\binom {k}{i-\frac{l-t}{2}}. \binom{i-\frac{l-t}{2}}{i-l+t} \hspace{50pt}\left( \text{by }\eqref{we15}\right)
\end{align*}
Since $H^{t}(J_0)$ is a pure Hodge structure, the Hodge number of $h^{r,s}(H^t(J_0))$ is
$$h^{r,s}(H^t(J_0))=\dim_{\mathbb{C}}\gr^{r}_{F}\gr^{s}_{\bar{F}}(H^{t}(J_0))=\binom{g-k}{r}\binom{g-k}{s},$$
where $r+s=t$. 
Taking sum over all such $0\leq t\leq l$ such that $l= t(\text{ mod } 2)$, we obtain the mixed Hodge number of type $(p,q)$ such that $p+q=l$ where $p=r+\frac{l-t}{2}$ and $q=s+\frac{l-t}{2}$. 
Therefore 
\[\dim_{\mathbb{C}}\gr_F^p\gr^p_{\bar{F}}\left(gr_{l}^{W}\left(H^{i}(\bar{J}_k)\right)\right)=\sum_{0\leq t\leq l, (l-t)~ \text{is even}} \binom {g-k}{p-\frac{l-t}{2}}\binom{g-k}{q-\frac{l-t}{2}}\binom {k}{i-\frac{l-t}{2}}. \binom{i-\frac{l-t}{2}}{i-l+t}.\]
\end{proof}

\end{document}